\documentclass{article}

\newif\ifsimax
\simaxfalse
\newif\ifjcam
\jcamfalse

\usepackage{amsmath,amssymb,amsthm,bm}
\usepackage[a4paper]{geometry}
\usepackage{algorithm,algorithmic}
\usepackage{graphicx}
\usepackage[colorlinks=false,pdfborder={0 0 0}]{hyperref}
\usepackage{xcolor}
\usepackage{authblk}
\usepackage{multirow}
\usepackage{nicematrix}
\usepackage{arydshln}

\DeclareMathOperator{\diag}{diag}

\DeclareMathOperator{\rank}{rank}

\DeclareMathOperator{\Span}{span}
\DeclareMathOperator{\trace}{trace}

\newcommand*{\set}[1]{\lbrace#1\rbrace}
\newcommand*{\SET}[2]{\left\lbrace#1\colon #2\right\rbrace}

\newcommand*{\trans}{^{\top}}
\newcommand*{\herm}{^*}

\newcommand{\Cases}[1]{\begin{cases}#1\end{cases}}

\newcommand{\bmat}[1]{\begin{bmatrix}#1\end{bmatrix}}

\newcommand*{\md}{\mathop{}\mathopen{}\mathrm{d}}

\newcommand*{\mi}{\mathrm i}

\def\adots{\mathinner{\mkern2mu\raise1pt\hbox{.}\mkern2mu
    \raise4pt\hbox{.}\mkern2mu\raise7pt\hbox{.}\mkern1mu}}

\newcommand*{\Expect}{\mathbb E}

\newcommand*{\new}{\mathrm{new}}
\newcommand*{\proj}{\mathrm{p}}
\newcommand*{\RR}{\mathrm{RR}}
\newcommand*{\Inner}{\mathrm{in}}
\newcommand*{\Outer}{\mathrm{out}}
\newcommand*{\NULL}{\mathrm{null}}
\DeclareMathOperator{\nnz}{nnz}
\newcommand*{\kSVD}{k_{\mathrm{SVD}}}
\newcommand*{\tildekSVD}{\tilde k_{\mathrm{SVD}}}
\newcommand*{\kGSVD}{k_{\mathrm{GSVD}}}
\newcommand*{\tildekGSVD}{\tilde k_{\mathrm{GSVD}}}
\newcommand*{\tol}{\mathtt{tol}}
\newcommand*{\macheps}{\bm u}

\DeclareMathOperator{\spec}{spec}

\usepackage[normalem]{ulem}

\definecolor{bkgndcolor}{rgb}{1,1,1}

\graphicspath{{fig/}}

\newenvironment{keywords}{\medskip\textbf{Keywords:}}{}

\newenvironment{MSCcodes}{\medskip\textbf{AMS subject classifications (2020).}}{}

\newtheorem{theorem}{Theorem}

\newtheorem{lemma}[theorem]{Lemma}

\title{A Contour Integral-Based Algorithm for Computing Generalized Singular
Values}
\author[1]{Yuqi Liu}
\author[2]{Xinyu Shan}
\author[2,3]{Meiyue Shao}
\affil[1]{School of Mathematical Sciences, Fudan University, Shanghai 200433,
China}
\affil[2]{School of Data Science, Fudan University, Shanghai 200433, China}
\affil[3]{MOE Key Laboratory for Computational Physical Sciences, Fudan
University, Shanghai 200433, China}
\date{\today}

\begin{document}
\pagecolor{bkgndcolor}

\maketitle

\begin{abstract}
We propose a contour integral-based algorithm for computing a few singular
values of a matrix or a few generalized singular values of a matrix pair.
Mathematically, the generalized singular values of a matrix pair are the
eigenvalues of an equivalent Hermitian--definite matrix pencil, known as the
Jordan--Wielandt matrix pencil.
However, direct application of the FEAST algorithm does not fully exploit the
structure of this problem.
We analyze several projection strategies on the Jordan--Wielandt matrix pencil,
and propose an effective and robust scheme tailored to GSVD.
Both theoretical analysis and numerical experiments demonstrate that our
algorithm achieves rapid convergence and satisfactory accuracy.
\end{abstract}



\section{Introduction}
\label{sec:introduction}
The generalized singular value decomposition (GSVD) of a matrix pair
proposed by Van Loan in~\cite{VanLoan1976} is a generalization of the singular
value decomposition (SVD) of a single matrix.
It was further developed by Paige and Saunders~\cite{PS1981} and many
others~\cite{BZ1993,PE1993,VanLoan1985}.
GSVD is closely related to many numerical linear algebra and data science
problems, such as the solution of eigenvalue problems~\cite{Betcke2008,EL1989,
Kagstrom1984,KYW2000}, several variants of least squares
problems\cite{VanLoan1976,Hansen1989,Paige1985}, 
the general Gauss--Markov linear model~\cite{EL1989,Paige1985}, the linear 
discriminant analysis~\cite{PP2005}, information retrieval~\cite{HJP2003}, 
real-time signal processing~\cite{NNI2012,SV1984}, etc.
GSVD has been applied to solve several real world problems, e.g., the
comparative analysis of the DNA microarrays~\cite{ABB2003} and ionospheric
tomography~\cite{BSB2004}.

There exist several numerical algorithms for computing the GSVD for dense
matrices~\cite{VanLoan1976,BD1993,Paige1986}.
Recent development on the stable computation of the CS
decomposition~\cite{Sutton2012} provides another alternative for computing the
GSVD.
As for large and sparse problems, there are also several approaches.
Zha's algorithm~\cite{Zha1996}, which is based on the CS decomposition and the
Lanczos bidiagonalization process, can be used to compute a few extreme
generalized singular values.
Jacobi--Davidson (JD) type algorithms~\cite{Hochstenbach2009,HJ2023} are
capable of computing generalized singular values at an arbitrary location.
In practice JD type algorithms are mostly used to find interior generalized
singular values.
Other popular symmetric eigensolvers, such as the LOBPCG
algorithm~\cite{Knyazev2001} and the ChASE algorithm~\cite{WSD2019}, can also
be adapted to GSVD because GSVD is essentially a symmetric eigenvalue
problem.

Contour integration is a recently developed technique for solving eigenvalue
problems, especially for finding interior eigenvalues.
The Sakurai--Sugiura (SS) method~\cite{SS2003} solves a symmetric 
eigenvalue problem by constructing a small moment matrix.
This method sometimes suffers from numerical instability since the moment
matrix, which is a Hankel matrix, is often ill-conditioned~\cite{SS2007}.
FEAST~\cite{Polizzi2009,TP2014} and CIRR~\cite{SS2007} belong to another
type of contour integral-based eigensolvers, which are essentially subspace
iteration applied to spectral projectors.
The technique of contour integration has been extended to solve other
eigenvalue problems, including non-Hermitian eigenvalue
problems~\cite{YXCCB2017}, polynomial eigenvalue problems~\cite{ASTIK2010},
and more general nonlinear eigenvalue problems~\cite{YS2013}.
Contour integral-based eigensolvers require solving a number of shifted linear
systems with multiple right-hand-sides.
Efficient methods for solving these shifted systems can be found in, e.g.,
\cite{FSFI2013,OKTS2010,SFT2013}.

In this work we discuss how to use contour integration to compute a few
singular values of a matrix \(A\) or a few generalized singular values of a
matrix pair \((A,B)\) within a given interval.
Since SVD and GSVD are special cases of the symmetric eigenvalue problem, a
straightforward approach is to apply the FEAST algorithm to either the Gram
matrix (pencil) or the Jordan--Wielandt matrix (pencil).
However such a naive approach is \emph{not} a good choice as it does not make
use of the special structure of SVD/GSVD.
We solve this problem by designing a structured FEAST algorithm tailored to
the Jordan--Wielandt matrix (pencil) so that it is both faster and more robust
compared to the naive approach.
We shall focus on how to choose an effective and robust projection scheme that
properly handles user-supplied initial guesses.
Without loss of generality, we assume that \(B\) 
has full column rank.
This assumption is usually plausible in practice as long as the null space is
properly deflated, possibly by another extremal eigensolver~\cite{FD2013}.

The rest of this paper is organized as follows.
We first briefly review some basic knowledge about GSVD and the FEAST
algorithm in Section~\ref{sec:preliminary}.
In Section~\ref{sec:algorithm-SVD} we propose a contour integral-based
algorithm for computing partial SVD.
Several projection schemes are discussed in detail.
The algorithm naturally carries over to GSVD, which is discussed in
Section~\ref{sec:algorithm-GSVD}.
In Section~\ref{sec:experiments} we present experimental results to
demonstrate the effectiveness of our algorithm.
The paper is concluded in Section~\ref{sec:conclusions}.

\section{Preliminaries}
\label{sec:preliminary}

\subsection{Generalized singular value decomposition}
\label{subsec:GSVD}
We first briefly review the generalized singular value decomposition of two
matrices \(A\in\mathbb C^{m\times n}\) and \(B\in\mathbb C^{p\times n}\) with \(p\geq n\).
For simplicity, we assume that \(B\) has full column rank, i.e.,
\(\rank(B)=n\).
There exist two unitary matrices
\(U\in\mathbb C^{m\times m}\),
\(V\in\mathbb C^{p\times p}\), and a nonsingular matrix
\(X\in\mathbb C^{n\times n}\) such that
\begin{equation}\label{eq:GSVD}
\begin{aligned}
&U\herm AX=\Sigma_A=
\begin{bNiceMatrix}[first-row,last-col]
q & n-q & \\
C & 0 & ~q \\
0 & 0 & ~m-q \\
\end{bNiceMatrix}\quad,\\
&V\herm BX=\Sigma_B=\begin{bNiceMatrix}[first-row,last-col]
q & n-q & \\
S & 0 & ~q \\
0 & I & ~n-q \\
0 & 0 & ~p-n
\end{bNiceMatrix}\quad,
\end{aligned}
\end{equation}
where 
\(C=\diag\set{\alpha_1,\alpha_2,\dotsc,\alpha_q}\) and
\(S=\diag\set{\beta_1,\beta_2,\dotsc,\beta_q}\) are positive diagonal
matrices satisfying \(C^2+S^2=I_q\).
The decomposition~\eqref{eq:GSVD} is known as the \emph{generalized singular
value decomposition (GSVD)} of \((A,B)\)~\cite{VanLoan1976}.
The formulation used here is a compressed form of the GSVD; see, e.g.,
\cite{Hansen1989}.
When \(B=I_n\), the GSVD reduces to the usual SVD of \(A\).

Partition \(U\), \(V\), and \(X\) as \(U=[u_1,u_2,\dotsc,u_m]\), 
\(V=[v_1,v_2,\dotsc,v_p]\), and
\(X=[x_1,x_2,\dotsc,x_n]\), respectively.
Then the \(5\)-tuple
\((\alpha_i,\beta_i,u_i,v_i,x_i)\), \(i=1\), \(\dotsc\), \(q\) is called a
\emph{nontrival GSVD component} of
\((A,B)\)~\cite{HJ2023,HJ2021}.
The number \(\sigma_i=\alpha_i/\beta_i\) is called a \emph{nontrival generalized
singular value};
\(u_i\) and \(v_i\) are the corresponding \emph{left generalized singular
vectors};
\(x_i\) is the corresponding \emph{right generalized singular vector}.

The nontrival GSVD components can be computed either from
\begin{equation}
\label{eq:GEP1}
A\herm Ax_i=B\herm Bx_i\sigma_i^2
\end{equation}
or
\begin{equation}
\label{eq:GEP2-A}
\bmat{ & A \\ A\herm & }\bmat{u_i \\ w_i}
=\bmat{I & \\ & B\herm B}\bmat{u_i \\ w_i}
\sigma_i,
\end{equation}
where \(w_i=x_i/\beta_i\), \(i=1,\ldots,q\).
Numerically, \eqref{eq:GEP2-A} is often superior to~\eqref{eq:GEP1} since the
cross-product \(A\herm A\), which can largely increase the condition number,
is avoided.
In the case that \(B\herm B\) is ill-conditioned,
\begin{equation}
\label{eq:GEP2-B}
\bmat{ & B \\ B\herm & }
\bmat{v_i \\ x_i/\alpha_i}
=\bmat{I & \\ & A\herm A}
\bmat{v_i \\ x_i/\alpha_i}\sigma_i^{-1}
\end{equation}
is another alternative.
Finally, we remark that in~\eqref{eq:GEP2-A} and~\eqref{eq:GEP2-B} the
eigenvalues appear in pairs.
For instance, \eqref{eq:GEP2-A} can be reformulated as
\[
\bmat{ & A \\ A\herm & }\bmat{u_i & u_i \\ w_i & -w_i}
=\bmat{I & \\ & B\herm B}\bmat{u_i & u_i \\ w_i & -w_i}
\bmat{\sigma_i & \\ & -\sigma_i}.
\]

\subsection{The FEAST algorithm}
\label{subsec:FEAST}
FEAST~\cite{Polizzi2009,TP2014} is a popular contour integral-based
eigensolver for computing all the eigenvalues in a given domain on the complex
plane.
In the following we briefly review the FEAST algorithm for symmetric
eigenvalue problems.

Let \(H\) be an \(n\times n\) Hermitian matrix with spectral decomposition
\[
H=\sum_{i=1}^n\lambda_iq_iq_i\herm,
\]
where \(\lambda_i\)'s are the eigenvalues of \(H\), and \(q_i\)'s are the
corresponding normalized eigenvectors.
Let us consider a domain \(\Omega\subset\mathbb C\) that encloses the
eigenvalues \(\lambda_{s+1}\leq\lambda_{s+2}\leq\dotsb\leq\lambda_{s+k}\).
Then the corresponding \emph{spectral projector} is given by
\[
P_{\Omega}(H)
=\sum_{i=1}^kq_{s+i}q_{s+i}\herm
=\frac{1}{2\pi\mi}\int_{\partial\Omega}(\xi I_n-H)^{-1}\md\xi.
\]
By applying one step of subspace iteration on \(P_{\Omega}(H)\), a basis of
the invariant subspace of \(H\) with respect to the eigenvalues
\(\lambda_{s+1}\), \(\dotsc\), \(\lambda_{s+k}\) can be extracted from the
columns of
\begin{equation}
\label{eq:spectral_projector}
P_{\Omega}(H)Z
=\frac{1}{2\pi\mi}\int_{\partial\Omega}(\xi I_n-H)^{-1}Z\md\xi
\end{equation}
for almost any \(Z\in\mathbb C^{n\times k}\).%
\footnote{By \emph{almost} we mean the set of \(Z\) that violates this
statement has Lebesgue measure zero.}
Then these eigenvalues can be easily calculated by the Rayleigh--Ritz
projection scheme.

FEAST is an algorithm that makes use of~\eqref{eq:spectral_projector} to
compute the eigenvalues within \(\Omega\).
In practice, an approximate spectral projector \(\tilde P_{\Omega}(H)\) is
applied because a numerical quadrature rule is performed to evaluate the
right-hand-side of~\eqref{eq:spectral_projector}.
For instance, if \(\Omega=\SET{\xi\in\mathbb C}{\lvert\xi-\mu_0\rvert<\rho}\),
the trapezoidal rule applied to \(\partial\Omega\) yields
\begin{align*}
\tilde P_{\Omega}(H)Z=\frac1N\sum_{j=0}^{N-1}\omega_j(\xi_jI_n-H)^{-1}Z,
\end{align*}
where \(\xi_j\)'s are equally-spaced quadrature nodes on \(\partial\Omega\),
and \(\omega_j\)'s are the corresponding quadrature weights.
Then a number of shifted linear systems with multiple right-hand-sides need to
be solved;
see~\cite{FSFI2013,OKTS2010,SFT2013} for discussions on how to solve these
linear systems efficiently.
Since \(P_{\Omega}(H)\cdot Z\) is computed only approximately, a few steps
of subspace iteration may be needed to refine the solution.
In practice, an ellipse instead of a circle can be used as the contour to
achieve more rapid convergence, and the trapezoidal rule often has
satisfactory accuracy here since the integrand is a sufficiently smooth
periodic function; see, e.g., \cite{AT2015,GPTV2015,TW2014} for detailed discussions.
The Zolotarev quadrature is sometimes a good alternative to the trapezoidal
rule, especially when some unwanted eigenvalues are close to the
contour~\cite{GPTV2015}.

In reality we need to estimate the number of eigenvalues in \(\Omega\), which
is equal to \(\trace\bigl(P_{\Omega}(H)\bigr)\), as this number is often not
known in advance.
For small- to medium-size problems, the number of desired eigenvalues can be
determined by computing the positive index of inertia through the \(LDL\herm\)
decomposition of two shifted matrices~\cite{Parlett1998}.
For large-scale problems, stochastic trace
estimators~\cite{AT2011,DPS2016,FTS2010,GDL2022,Girard1989} are often used.
If estimating \(\trace\bigl(P_{\Omega}(H)\bigr)\) becomes expensive, an
alternative approach is to gradually double the dimension of 
the trial subspace, i.e., the number of columns of \(Z\), until a clear 
gap in the spectrum is detected.

The FEAST algorithm also carries over to generalized eigenvalue problems.
For instance, for a Hermitian--definite matrix pencil \((H,M)\), the
corresponding spectral projector becomes
\[
P_{\Omega}(H,M)
=\frac{1}{2\pi\mi}\int_{\partial\Omega}(\xi I_n-M^{-1}H)^{-1}\md\xi
=\frac{1}{2\pi\mi}\int_{\partial\Omega}(\xi M-H)^{-1}M\md\xi,
\]
which is a self-adjoint positive semidefinite operator in the \(M\)-inner
product.
Hence there is no essential difference compared to the usual symmetric
eigenvalue problem.

\section{A FEAST-SVD Algorithm}
\label{sec:algorithm-SVD}
We first discuss how to compute the (partial) singular value decomposition
using contour integration, as this is an important special case of GSVD.
We assume that the singular values of \(A\) that lie in the interval
\((\alpha,\beta)\) are of interest, with \(\beta>\alpha>0\).
In this section we discuss how to apply the FEAST algorithm to solve this
problem robustly.
Special care will be taken to properly handle user-supplied initial guesses,
since many real world applications naturally provide meaningful initial
guesses to eigensolvers.

\subsection{A naive approach}
\label{subsec:naive}
Finding \(A\)'s singular values in \((\alpha,\beta)\) is mathematically
equivalent to finding the eigenvalues of the Gram matrix \(A\herm A\) in
\((\alpha^2,\beta^2)\).
Therefore, we can apply the FEAST algorithm on \(A\herm A\) using a contour
that encloses \((\alpha^2,\beta^2)\).
However, the accuracy of this approach can be low when small singular values
of \(A\) are desired.
More accurate results can be expected if we compute the eigenvalues of the
Jordan--Wielandt matrix 
\[\check A=\bmat{& A \\ A\herm &}\]
instead.
The desired singular values correspond to \(\check A\)'s eigenvalues in
\((-\beta,-\alpha)\) \(\cup(\alpha,\beta)\).
A naive approach is to solve this symmetric eigenvalue problem using the FEAST
algorithm.
In the following we show that such a naive approach is \emph{not} robust with
respect to the initial guess.

By choosing \(a\geq(\beta-\alpha)/2\) and \(b>0\), the ellipse
\[
\Gamma^+
=\SET{\frac{\alpha+\beta}{2}+a\cos\theta+\mi b\sin\theta}{\theta\in[0,2\pi]}
\]
of semi-axes \(a\) and \(b\) encloses the interval \((\alpha,\beta)\);
see Figure~\ref{fig:ellipse}.
The corresponding spectral projector is
\[
P^+(\check A)
=\frac{1}{2\pi\mi}\int_{\Gamma^+}(\xi I_{m+n}-\check A)^{-1}\md\xi.
\]
If \(\tilde U\) and \(\tilde W\), respectively, contain approximate left and
right singular vectors of \(\check A\),%
\footnote{Matrices with tilde symbols are used to denote approximations to
certain exact values.}
we expect
\begin{equation}
\label{eq:proj+}
\bmat{\tilde U_{\new} \\ \tilde W_{\new}}
=P^+(\check A)\bmat{\tilde U \\ \tilde W}
\end{equation}
to deliver a basis of the desired subspace, as this can be viewed as one step
of subspace iteration.

\begin{figure}[tb!]
\centering
\includegraphics[scale=1]{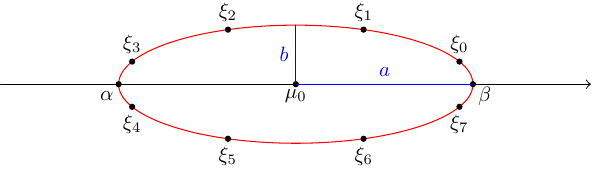}
\caption{An ellipse contour with \(a=(\beta-\alpha)/2\) and eight quadrature nodes (i.e., \(N=8\)).}
\label{fig:ellipse}
\end{figure}

However, we remark that~\eqref{eq:proj+} is \emph{not} always a good choice,
as \([\tilde U_{\new}\herm,\tilde W_{\new}\herm]\herm\) can be rank deficient.
Let us consider a matrix \(A\in\mathbb{C}^{m\times n}\) \((m>n)\) with its full SVD being partitioned as
\[
A=[U_\Inner,U_\Outer,U_\NULL]
\bmat{\Sigma_\Inner & \\ & \Sigma_\Outer \\ 0 & 0}[W_\Inner,W_\Outer]\herm,
\]
where the diagonal elements of \(\Sigma_\Inner\) are exactly the desired
singular values.
Then we can express \([\tilde U\herm,\tilde W\herm]\herm\) as
\[
\bmat{\tilde U \\ \tilde W}
=\frac{1}{\sqrt2}\left(\bmat{U_\Inner \\ W_\Inner}C_\Inner^+
+\bmat{U_\Outer \\ W_\Outer}C_\Outer^+
+\bmat{U_\Inner \\ -W_\Inner}C_\Inner^-
+\bmat{U_\Outer \\ -W_\Outer}C_\Outer^-\right)
+\bmat{U_\NULL \\ 0}C_\NULL.
\]
Applying \(P^+(\check A)\) yields
\[
\bmat{\tilde U_{\new} \\ \tilde W_{\new}}
=P^+(\check A)\bmat{\tilde U \\ \tilde W}
=\frac{1}{\sqrt2}\bmat{U_\Inner \\ W_\Inner}C_\Inner^+,
\]
so that severe cancellation can occur when
\(\lVert C_\Inner^+\rVert\ll\lVert C_\Inner^-\rVert\).
An extreme case is that \(\lVert C_\Inner^+\rVert=0\) while
\(\lVert C_\Inner^-\rVert=\Theta(1)\), which leads to
\[
\bmat{\tilde U_{\new} \\ \tilde W_{\new}}
=P^+(\check A)\bmat{\tilde U \\ \tilde W}
=0
\]
because \([\tilde U\herm,\tilde W\herm]\herm\) belongs to an invariant
subspace of \(\check A\) that is orthogonal to the desired one.
This causes the algorithm to break down in exact arithmetic, or converge
slowly in practice by taking into account quadrature and rounding errors.
This issue needs to be tackled as we would like to develop a robust solver.

\subsection{Structured Galerkin condition and Rayleigh--Ritz projection}
\label{subsec:SVD-Galerkin}
Clearly, the naive approach in Section~\ref{subsec:naive} does not make use of
the structure of \(\check A\).
In order to exploit the symmetry of \(\check A\), we first derive a structured
Galerkin condition and propose a tailored Rayleigh--Ritz projection scheme.

\subsubsection{Structured Galerkin condition}
Suppose that \([\tilde U\herm,\tilde W\herm]\herm\) contains approximate
eigenvectors of \(\check A\) with
\(\tilde U\herm\tilde U=\tilde W\herm\tilde W=I\),
and~\(\tilde\Sigma\) is a diagonal matrix whose diagonal entries consist of
the corresponding approximate eigenvalues.
Then the standard Galerkin condition reads
\begin{equation}
\label{eq:SVD-Galerkin}
\Span({R})\perp\Span\left(\bmat{\tilde U \\ \tilde W}\right),
\end{equation}
where
\[
R=\check A\bmat{\tilde U \\ \tilde W}-\bmat{\tilde U \\ \tilde W}\tilde\Sigma
\]
is the (block) residual and the symbol \(\perp\) means the
orthogonality between two subspaces.
Since the nonzero eigenvalues of \(\check A\) appear in pairs \(\pm\sigma\),
a \emph{structured} Galerkin
condition is
\begin{equation}
\label{eq:SVD-Galerkin2}
\Span({\check R})\perp\Span\left(\bmat{\tilde U & \tilde U \\ \tilde W & -\tilde W}\right),
\end{equation}
where
\[
\check R=\check A\bmat{\tilde U & \tilde U \\ \tilde W & -\tilde W}
-\bmat{\tilde U & \tilde U \\ \tilde W & -\tilde W}
\bmat{\tilde\Sigma \\ & -\tilde\Sigma}
\]
is the structured residual.
Clearly~\eqref{eq:SVD-Galerkin2} is stronger than~\eqref{eq:SVD-Galerkin},
under the assumption that \(\tilde\Sigma\) is positive definite.

\subsubsection{Rayleigh--Ritz projection}
The condition~\eqref{eq:SVD-Galerkin2} simplifies to
\(\tilde U\herm A\tilde W=\tilde\Sigma\), or equivalently,
\[
\Cases{\Span(A\tilde W-\tilde U\tilde\Sigma)\perp\Span(\tilde U),\\
\Span(A\herm\tilde U-\tilde W\tilde\Sigma)\perp\Span(\tilde W).}
\]
A Rayleigh--Ritz projection scheme can then be derived.
Suppose that we are given \(\tilde U_0\in\mathbb C^{m\times k}\) and
\(\tilde W_0\in\mathbb C^{n\times k}\) with
\(\tilde U_0\herm\tilde U_0=\tilde W_0\herm\tilde W_0=I_k\).
Let the (full) singular value decomposition of
\(A_{\proj}=\tilde U_0\herm A\tilde W_0\) be
\[
A_{\proj}=\tilde U_{\proj}\tilde\Sigma\tilde W_{\proj}\herm.
\]
By setting \(\tilde U_{\RR}\gets\tilde U_0\tilde U_{\proj}\),
\(\tilde W_{\RR}\gets\tilde W_0\tilde W_{\proj}\), we obtain
\(\tilde U_{\RR}\herm\tilde U_{\RR}=\tilde W_{\RR}\herm\tilde W_{\RR}=I_k\)
and
\[
\tilde U_{\RR}\herm A\tilde W_{\RR}=\tilde\Sigma,
\]
which satisfies the structured Galerkin condition~\eqref{eq:SVD-Galerkin2}.
This is essentially equivalent to the usual Petrov--Galerkin method for SVD.
We remark that
\[
\frac{1}{\sqrt2}
\bmat{\tilde U_{\RR} & \tilde U_{\RR} \\ \tilde W_{\RR} & -\tilde W_{\RR}}
\]
contains orthonormalized approximate eigenvectors of \(\check A\).
In this manner we are able to extract eigenvalues of \(\check A\) in pairs.

\subsection{Choice of the spectral projector}
\label{subsec:SVD-projector}

\subsubsection{A simple spectral projector}
\label{subsubsec:P^+}
The simplest spectral projector is to utilize the filter described by the
naive FEAST algorithm in Section~\ref{subsec:naive}.
The spectral projector is
\[
P^+(\check A)
=\frac{1}{2\pi\mi}\int_{\Gamma^+}(\xi I_{m+n}-\check A)^{-1}\md\xi.
\]
Assume that \(\tilde U\) and \(\tilde W\) contain approximate left and right
singular vectors of \(\check A\), respectively.
From the analysis in Section~\ref{subsec:naive}, it is evident that
\begin{equation*}
\bmat{\tilde U_{\new} \\ \tilde W_{\new}}
=P^+(\check A)\bmat{\tilde U \\ \tilde W}
\end{equation*}
can sometimes have severe cancellation.
Hence, this remains an unsuitable choice.

In practice cancellation can indeed occur if \(\tilde U\) and \(\tilde W\)
are constructed separately without proper post-processing.
Suppose \(\tilde U\approx U_\Inner Q_1\) and \(\tilde W\approx W_\Inner Q_2\)
are provided by the user, where \(Q_1\) and \(Q_2\) are unitary matrices.
Then we have \(\Span(\tilde U)\approx\Span(U_\Inner)\)
and \(\Span(\tilde W)\approx\Span(W_\Inner)\), indicating that both
\(\tilde U\) and \(\tilde W\) have good quality.
However, applying the simple projector \(P^+(\check A)\) to
\([\tilde U\herm,\tilde W\herm]\herm\) can cause the algorithm to break down
or converge slowly in practice, especially when \(Q_2\approx-Q_1\).
This issue needs to be tackled as we would like to develop a robust solver
that makes use of user-supplied initial guesses.

\subsubsection{A simple spectral projector with Rayleigh--Ritz projection}
\label{subsubsec:P^+_RR}
One way to resolve the potential issue of cancellation or rank deficiency is
to perform a Rayleigh--Ritz projection scheme in the first iteration of the
FEAST algorithm before applying the spectral projector.
Suppose that we update the orthogonal bases through
\(\tilde U_{\RR}\gets\tilde U\tilde U_{\proj}\) and
\(\tilde W_{\RR}\gets\tilde W\tilde W_{\proj}\), where
\(\tilde U_{\proj}\tilde\Sigma\tilde W_{\proj}\herm\) is the singular value
decomposition of \(\tilde U\herm A\tilde W\).
Let us represent \([\tilde U_{\RR}\herm,\tilde W_{\RR}\herm]\herm\) as
\[
\bmat{\tilde U_{\RR} \\ \tilde W_{\RR}}
=\frac{1}{\sqrt2}\left(\bmat{U \\ W}C^+_{\RR}+\bmat{U \\ -W}C^-_{\RR}\right)
+\bmat{U_\NULL \\ 0}C_\NULL.
\]
It can be verified that
\[
2\tilde\Sigma
=\bmat{\tilde U_{\RR} \\ \tilde W_{\RR}}\herm\check A
\bmat{\tilde U_{\RR} \\ \tilde W_{\RR}}
=(C_{\RR}^+)\herm\Sigma C_{\RR}^+-(C_{\RR}^-)\herm\Sigma C_{\RR}^-.
\]
Then we have
\[
\lVert C_{\RR}^+\rVert_{\Sigma}\geq\lVert C_{\RR}^-\rVert_{\Sigma},
\]
indicating that \(\lVert C_{\RR}^+\rVert\) is reasonably large.
Moreover, \(C_{\RR}^+\) is in general not rank deficient when \(\tilde\Sigma\)
is positive definite, since the smallest eigenvalue of
\((C_{\RR}^+)\herm\Sigma C_{\RR}^+\) is at least as large as
\(\lVert\tilde\Sigma^{-1}\rVert_2^{-1}\).
Thus, the issue of rank deficiency by applying~\eqref{eq:proj+} 
can be resolved by performing a Rayleigh--Ritz projection scheme.

However, in certain instances, performing a Rayleigh--Ritz projection can
yield \emph{worse} results.
To illustrate this, let us consider a small example.
Suppose that the largest singular value \(\sigma_1\) of a matrix
\(A\in\mathbb{C}^{4\times 4}\) is of interest.
Let \(u_i\) and \(w_i\), respectively, be the \(i\)th left and right singular
vectors of \(A\).
Suppose that the initial guess is given by
\[
U_0=[u_1,u_2]+[u_3,u_4]\bmat{\sigma_3 \\ & \sigma_4}^{-1}
W_{\proj}\Delta_\epsilon U_{\proj}\herm,
\qquad
W_0=[w_3,w_4],
\]
where
\[
U_{\proj}=\bmat{1 & 0 \\ 0 & 1},
\qquad
W_{\proj}=\frac{1}{\sqrt{2}}\bmat{1 & 1 \\ 1 & -1},
\qquad
\Delta_{\epsilon}=\bmat{\epsilon \\ & 0}.
\]
Performing a Rayleigh--Ritz procedure by means of \(U_0\) and \(W_0\) yields
\(U_0\herm AW_0=U_{\proj}\Delta_\epsilon W_{\proj}\herm\).
Then \(Z=[U_0\herm,W_0\herm]\herm\) is transformed to
\[
Z_{\RR}=\bmat{U_0U_{\proj} \\ W_0W_{\proj}}
=\bmat{u_1+\Delta_1 & u_2 \\
(w_3+w_4)/\sqrt2 & (w_3-w_4)/\sqrt2},
\]
with \(\lVert\Delta_1\rVert_2=O(\epsilon)\).
In practice \(\tilde P^+(A)Z_{\RR}\) can sometimes be much worse than
\(\tilde P^+(A)Z\), where \(\tilde P^+(A)\) is the approximate filter.

For instance, suppose \(A=\diag\set{0.93,0.92,0.91,0.09}\),
\(\tilde P^+(\sigma_1)=1\), \(\tilde P^+(\sigma_2)=0.9\),
\(\tilde P^+(\sigma_3)=0.8\), \(\tilde P^+(\sigma_4)=0.1\),
\(\tilde P^+(-\sigma_1)=\dotsb=\tilde P^+(-\sigma_4)=0\), and
\(\epsilon=10^{-16}\).
The relative residuals of the two approximate singular triplets using
\(\tilde P^+(A)Z\) are \(0.0053\) and \(0.0495\), respectively.
If \(\tilde P^+(A)Z_{\RR}\) is used instead, the relative residuals become
\(0.0269\) and \(0.0390\) --- the residual with respect to \(\sigma_1\) is
over \(5\times\) larger after performing the Rayleigh--Ritz procedure!

\subsubsection{A pair of contours}
\label{subsubsec:P^++P^-}
Another way to resolve the potential difficulty of rank deficiency when
applying~\eqref{eq:proj+} is to introduce a contour~\(\Gamma^-\) that is
symmetric to \(\Gamma^+\) with respect to the origin, i.e.,
\[
P^-(\check A)
=\frac{1}{2\pi\mi}\int_{\Gamma^-}(\xi I_{m+n}-\check A)^{-1}\md\xi
=\bmat{I_m \\ & -I_n}P^+(\check A)\bmat{I_m \\ & -I_n}.
\]
Then we expect
\[
P^+(\check A)+P^-(\check A)
=\frac{1}{2\pi\mi}\int_{\Gamma^+\cup\Gamma^-}
(\xi I_{m+n}-\check A)^{-1}\md\xi
\]
to be a safer projector compared to either \(P^+(\check A)\) or
\(P^-(\check A)\), since \(P^+(\check A)+P^-(\check A)\) can properly preserve
the useful information in \([\tilde U\herm,\tilde W\herm]\), regardless of the
sign of the Ritz values.

Though we introduce \(P^+(\check A)+P^-(\check A)\) in order to make the
filter robust, and expect that it is better than \(P^+(\check A)\),
we remark that \(P^+(\check A)+P^-(\check A)\) can sometimes be \emph{worse}
than \(P^+(\check A)\) in terms of convergence.
We use a simple example to illustrate this.
Suppose that we are computing the first eigenvector,
\([u_1\herm,w_1\herm]\herm\), from the initial guess
\[
\bmat{\tilde u \\ \tilde w}
=\frac{1}{\sqrt2}
\sum_{i=1}^n\bmat{u_i & u_i \\ w_i & -w_i}\bmat{\alpha_i^+ \\ \alpha_i^-}
+\bmat{U_\NULL \\ 0}c_\NULL
\]
using an approximate filter function \(\tilde P^+(\cdot)\) satisfying
\[
\tilde P^+(\mu)=\Cases{\Theta(1), & \text{if \(\mu\approx\sigma_1\)},\\
O(\epsilon), & \text{otherwise},}
\]
where \(\epsilon\) is a tiny positive number.
On the one hand, we have
\begin{align*}
\bmat{\tilde u_{\new}^{\pm} \\ \tilde w_{\new}^{\pm}}
&=\bigl(\tilde P^+(\check A)+\tilde P^-(\check A)\bigr)
\bmat{\tilde u \\ \tilde w}\\
&=\Theta(1)\cdot\bmat{u_1(\alpha_1^++\alpha_1^-) \\ w_1(\alpha_1^+-\alpha_1^-)}
+\sum_{i=2}^nO(\epsilon)\cdot
\bmat{u_i & u_i \\ w_i & -w_i}\bmat{\alpha_i^+ \\ \alpha_i^-}
+O(\epsilon)\bmat{U_\NULL \\ 0}c_\NULL.
\end{align*}

If cancellation occurs in either \(\alpha_1^++\alpha_1^-\) or
\(\alpha_1^+-\alpha_1^-\), which is not uncommon in practice, then it requires
additional effort to identify \(u_1\) and \(w_1\) since the remaining
\(O(\epsilon)\) terms have nonnegligible impact.
On the other hand,
\begin{align*}
\bmat{\tilde u_{\new}^+ \\ \tilde w_{\new}^+}
&=\tilde P^+(\check A)\bmat{\tilde u \\ \tilde w}\\
&=\Theta(1)\cdot \bmat{u_1 \\ w_1}\alpha_1^+
+O(\epsilon)\cdot \bmat{u_1 \\ -w_1}\alpha_1^-
+\sum_{i=2}^nO(\epsilon)\cdot
\bmat{u_i & u_i \\ w_i & -w_i}\bmat{\alpha_i^+ \\ \alpha_i^-}
+O(\epsilon)\bmat{U_\NULL \\ 0}c_\NULL
\end{align*}
is a reasonably good approximation to a multiple of
\([u_1\herm,w_1\herm]\herm\) as long as \(\lvert \alpha_1^+\rvert\) is not too
small.

Since we are only interested in one singular value in this example, it is
possible to compute the Rayleigh quotient explicitly.
We compute the following Rayleigh quotients:
\begin{equation}
\label{eq:RR-p++p-}
\begin{aligned}
\frac{(A\tilde w_{\new}^{\pm})\herm(A\tilde w_{\new}^{\pm})}
{(\tilde w_{\new}^{\pm})\herm\tilde w_{\new}^{\pm}}
&=\sigma_1^2-\frac{\sum_{i=2}^nO(\epsilon^2)(\sigma_1^2-\sigma_i^2)
(\alpha_i^+-\alpha_i^-)^2}
{\Theta(1)(\alpha_1^+-\alpha_1^-)^2
+\sum_{i=2}^nO(\epsilon^2)(\alpha_i^+-\alpha_i^-)^2},\\
\frac{(A\herm\tilde u_{\new}^{\pm})\herm(A\herm\tilde u_{\new}^{\pm})}
{(\tilde u_{\new}^{\pm})\herm\tilde u_{\new}^{\pm}}
&=\sigma_1^2-\frac{\sum_{i=2}^nO(\epsilon^2)(\sigma_1^2-\sigma_i^2)
(\alpha_i^++\alpha_i^-)^2+O(\epsilon^2)\sigma_1^2c_\NULL\herm c_\NULL}
{\Theta(1)(\alpha_1^++\alpha_1^-)^2
+\sum_{i=2}^nO(\epsilon^2)(\alpha_i^++\alpha_i^-)^2+O(\epsilon^2)c_\NULL\herm c_\NULL},
\end{aligned}
\end{equation}
and
\begin{equation}
\label{eq:RR-p+}
\begin{aligned}
\frac{(A\tilde w_{\new}^+)\herm(A\tilde w_{\new}^+)}
{(\tilde w_{\new}^+)\herm\tilde w_{\new}^+}
&=\sigma_1^2-\frac{\sum_{i=2}^nO(\epsilon^2)(\sigma_1^2-\sigma_i^2)
(\alpha_i^+-\alpha_i^-)^2}
{\bigl(\Theta(1)\alpha_1^+-O(\epsilon)\alpha_1^-\bigr)^2
+\sum_{i=2}^nO(\epsilon^2)(\alpha_i^+-\alpha_i^-)^2},\\
\frac{(A\herm\tilde u_{\new}^+)\herm(A\herm\tilde u_{\new}^+)}
{(\tilde u_{\new}^+)\herm\tilde u_{\new}^+}
&=\sigma_1^2-\frac{\sum_{i=2}^nO(\epsilon^2)(\sigma_1^2-\sigma_i^2)
(\alpha_i^++\alpha_i^-)^2+O(\epsilon^2)\sigma_1^2c_\NULL\herm c_\NULL}
{\bigl(\Theta(1)\alpha_1^++O(\epsilon)\alpha_1^-\bigr)^2
+\sum_{i=2}^nO(\epsilon^2)(\alpha_i^++\alpha_i^-)^2
+O(\epsilon^2)c_\NULL\herm c_\NULL}.
\end{aligned}
\end{equation}
Using the monotonicity of linear fractional functions, we easily see
that~\eqref{eq:RR-p+} is better when cancellation occurs in
either \(\alpha_1^++\alpha_1^-\) or \(\alpha_1^+-\alpha_1^-\),
while~\eqref{eq:RR-p++p-} is better when \(\lvert\alpha_1^+\rvert\) is tiny.
Since cancellation in \(\alpha_1^++\alpha_1^-\) or \(\alpha_1^+-\alpha_1^-\)
is not uncommon, especially for real matrices, the use of~\eqref{eq:RR-p++p-}
can potentially be problematic.

\subsubsection{An augmented scheme with a pair of contours}
\label{subsubsec:P^+andP^-}
We have seen that using \(P^+(\check A)\) alone can be problematic since
\([\tilde U\herm,\tilde W\herm]\herm\) can contain eigenvectors
of \(\check A\) with respect to negative eigenvalues so that
\(P^+(\check A)[\tilde U\herm,\tilde W\herm]\herm\) is rank deficient.
Using \(P^+(\check A)+P^-(\check A)\) is not fully satisfactory for our
purpose as it can sometimes be even worse than \(P^+(\check A)\).
Another possibility is to compute
\begin{equation}
\label{eq:proj+pair0}
\bmat{\tilde U_{\new} \\ \tilde W_{\new}}
=\bmat{P^+(\check A)\bmat{\tilde U \\ \tilde W},
P^-(\check A)\bmat{\tilde U \\ \tilde W}}
\end{equation}
and then use \(\tilde U_{\new}\) and \(\tilde W_{\new}\) for the
Rayleigh--Ritz projection.
Note that
\[
P^-(\check A)\bmat{\tilde U \\ \tilde W}
=\bmat{I_m \\ & -I_n}P^+(\check A)
\bmat{\tilde U \\ -\tilde W}.
\]
We can use
\begin{equation}
\label{eq:proj+pair}
\bmat{\tilde U_{\new} \\ \tilde W_{\new}}
=P^+(\check A)\bmat{\tilde U & \tilde U \\ \tilde W & -\tilde W}
\end{equation}
instead of~\eqref{eq:proj+pair0}.
In this manner all useful information from \(P^+(\check A)\) and
\(P^-(\check A)\) is kept, while the price to pay is that the size of 
the projected problem is doubled.

Fortunately, we do not always have to double the problem size.
The issue of using \(P^+(\check A)\) alone is avoided by
applying~\eqref{eq:proj+pair} since \(P^-(\check A)\) is also 
implicitly involved.
Starting from the second iteration, we can ensure that components with respect
to positive eigenvalues of \(\check A\) are already encoded in
\([\tilde U\herm,\tilde W\herm]\herm\).
Therefore, we only need to apply~\eqref{eq:proj+pair} in the first iteration.
Then in subsequent iterations applying~\eqref{eq:proj+} already suffices.
In fact, such a combination of~\eqref{eq:proj+pair} and~\eqref{eq:proj+} can
dramatically accelerate the convergence.
We shall illustrate this phenomenon by numerical experiments in
Section~\ref{sec:experiments}, and provide a brief explanation in the
Appendix.
Finally, we summarize this strategy as Algorithm~\ref{alg:FEAST-SVD}.

\begin{algorithm}
\caption{A FEAST-SVD algorithm.}
\label{alg:FEAST-SVD}
\begin{algorithmic}[1]
\REQUIRE A matrix \(A\in\mathbb{C}^{m\times n}\);
location of desired singular values \((\alpha,\beta)\);
initial guess \(U_0\in \mathbb{C}^{m\times \ell}\) and
\(W_0\in \mathbb{C}^{n\times\ell}\) satisfying \(U_0\herm U_0=W_0\herm W_0=I_\ell\);
quadrature nodes with weights \((\xi_j,\omega_j)\)'s for \(j=1\), \(2\),
\(\dotsc\), \(N\).
The number of desired singular values \(k\) cannot exceed \(\ell\).
\ENSURE Approximate singular triples \((\Sigma,U,W)\).
\STATE \(U\gets U_0\), \(W\gets W_0\).
\FOR{\({\mathtt{iter}}=1\), \(2\), \(\dotsc\)}
  \IF{\({\mathtt{iter}}=1\)}
    \STATE \(Z\gets\bmat{U & U \\ W & -W}\).
  \ELSE
    \STATE \(Z\gets\bmat{U \\ W}\).
  \ENDIF
  \STATE \([U\herm,W\herm]\herm\gets\sum_{j=1}^N\omega_j(\xi_jI-\check A)^{-1}Z\).
  \STATE Orthogonalize \(U\) and \(W\) so that \(U\herm U=W\herm W=I\).
  \STATE Solve SVD of \(A_{\proj}=U\herm AW\) to obtain the singular triplet
    \((\Sigma,U_{\proj},W_{\proj})\).
  \STATE \(U\gets UU_{\proj}\), \(W\gets WW_{\proj}\).
  \STATE Check convergence.
  \IF{\({\mathtt{iter}}=1\)}
    \STATE Choose \(\ell\) best components from \((\Sigma,U,W)\) as the new
      \((\Sigma,U,W)\).
    \label{alg-step:chooseL}
  \ENDIF
\ENDFOR
\end{algorithmic}
\end{algorithm}

It is worth mentioning that in Step~\ref{alg-step:chooseL} of
Algorithm~\ref{alg:FEAST-SVD}, there are \(2\ell\) SVD triplets after the
Rayleigh--Ritz projection, and we need to choose \(\ell\) of them.
The singular values are usually chosen to be the closest to the desired
interval.
If there are more than \(\ell\) Ritz values lying in the interval, which is not
uncommon in practice, the SVD triplets with smaller residuals will be chosen.

Once a few singular triplets have attained satisfactory accuracy, it is
sensible to deflate these singular triplets to reduce the cost in subsequent
calculations.
The following \emph{soft locking} strategy~\cite{KALO2007} can be adopted:
the Rayleigh--Ritz projection involves all approximate singular vectors, while
spectral projectors are only applied to undeflated singular vectors.

\subsection{Trace estimation}
\label{subsec:trace-SVD}
In practice, we need to estimate the number of singular values \(k\) that lie
in the desired interval \((\alpha,\beta)\) before applying
Algorithm~\ref{alg:FEAST-SVD}.
In Section~\ref{subsec:FEAST}, we already mentioned that \(k\) is equal to the
trace of the spectral projector.
Once an estimate of \(k\) is available, we can choose an appropriate initial
guess with sufficiently many columns.

In order to estimate \(k\) we adopt the Monte Carlo trace
estimator~\cite{Girard1989}.
Notice that
\[
P^+(\check A)
=\frac{1}{2\pi\mi}\int_{\Gamma^+}(\xi I_{m+n}-\check A)^{-1}\md\xi
\]
is Hermitian and positive semidefinite.
Let \(y\in\mathbb R^{m+n}\) be a random vector satisfying
\(\Expect[yy\herm]=I_{m+n}\).
Then we have
\[
\trace\bigl(P^+(\check A)\bigr)=\Expect\bigl[y\herm P^+(\check A)y\bigr].
\]
Thus the problem reduces to calculating
\(\hat y=P^+(\check A)y\approx\sum_{j=0}^{N-1}\omega_j(\xi_jI-\check A)^{-1}y/N\).
In practice we can draw several independent samples \(y_1\), \(y_2\),
\(\dotsc\), \(y_K\) of \(y\) to obtain a reasonably accurate estimate of
\(\trace\bigl(P^+(\check A)\bigr)\) through
\[
\trace\bigl(P^+(\check A)\bigr)
\approx\frac1K\sum_{i=1}^Ky_i\herm\hat y_i.
\]

Finally, we remark that when computing \(\hat y\), the number of quadrature
nodes \(N\) does not need to be the same as that in
Algorithm~\ref{alg:FEAST-SVD}.
Sometimes a smaller number of quadrature nodes suffices in the trace
estimator.

 \section{A FEAST-GSVD Algorithm}
\label{sec:algorithm-GSVD}
Similar to (partial) singular value decomposition, the computation of
generalized singular value decomposition can be reduced to a generalized
symmetric eigenvalue problem through~\eqref{eq:GEP1}, \eqref{eq:GEP2-A}
or~\eqref{eq:GEP2-B}.
We assume that \(B\) is moderately well-conditioned so that~\eqref{eq:GEP2-A}
is an appropriate choice.
In the following we discuss how to derive a FEAST algorithm for the
\emph{Jordan--Wielandt matrix pencil}
\[
(\check A,\check B)
=\left(\bmat{ & A \\ A\herm & },\bmat{I & \\ & B\herm B}\right).
\]

\subsection{Rayleigh--Ritz projection}
\label{sec:GSVD-projector}
Let \([\tilde U\herm,\tilde W\herm]\herm\) consist of approximate
eigenvectors of \((\check A,\check B)\), and \(\tilde\Sigma\) be a positive
definite diagonal matrix whose diagonal entries are the corresponding
approximate eigenvalues.
The standard Galerkin condition becomes
\begin{equation*}
\label{eq:GSVD-Galerkin}
\Span\left(\check A\bmat{\tilde U \\ \tilde W}
-\check B\bmat{\tilde U \\ \tilde W}\tilde\Sigma\right)
\perp\Span\left(\bmat{\tilde U \\ \tilde W}\right).
\end{equation*}
Then a structured Galerkin condition is
\begin{equation*}
\label{eq:GSVD-Galerkin2}
\Span\left(\check A\bmat{\tilde U & \tilde U \\ \tilde W & -\tilde W}
-\check B\bmat{\tilde U & \tilde U \\ \tilde W & -\tilde W}
\bmat{\tilde\Sigma \\ & -\tilde\Sigma}\right)
\perp\Span\left(\bmat{\tilde U & \tilde U \\ \tilde W & -\tilde W}\right),
\end{equation*}
which simplifies to
\begin{equation}
\label{eq:GSVD-Galerkin2-simplified}
\Cases{
\Span(A\tilde W-\tilde U\tilde\Sigma)\perp\Span(\tilde U),\\
\Span(A\herm\tilde U-B\herm B\tilde W\tilde\Sigma)\perp\Span(\tilde W).}
\end{equation}

To derive a Rayleigh--Ritz projection scheme based
on~\eqref{eq:GSVD-Galerkin2-simplified},
we assume that \(\tilde U_0\in\mathbb C^{m\times k}\) and
\(\tilde W_0\in\mathbb C^{n\times k}\) contain orthonormalized columns in the
sense that
\[
\tilde U_0\herm\tilde U_0=I_k,
\qquad\tilde W_0\herm B\herm B\tilde W_0=I_k.
\]
Let the singular value decomposition of
\(A_{\proj}=\tilde U_0\herm A\tilde W_0\) be
\[
A_{\proj}=\tilde U_{\proj}\tilde\Sigma\tilde W_{\proj}\herm.
\]
By setting \(\tilde U_{\RR}\gets\tilde U_0\tilde U_{\proj}\),
\(\tilde W_{\RR}\gets\tilde W_0\tilde W_{\proj}\), we obtain
\[
\tilde U_{\RR}\herm A\tilde W_{\RR}=\tilde\Sigma,
\qquad\tilde U_{\RR}\herm\tilde U_{\RR}=I_k,
\qquad\tilde W_{\RR}\herm B\herm B\tilde W_{\RR}=I_k.
\]

The other approximate left generalized singular vectors of \((A, B)\) can be
obtained by
\[
\tilde V=B\tilde W_{\RR}.
\]
Computing \(\tilde V\) from \(\tilde W_{\RR}\) is reasonably accurate
here, since we assume that \(B\) is moderately well-conditioned.
If, in addition, \(X\), \(C\), and \(S\) are needed, they can be computed
through
\[
\qquad\tilde S=(I+\tilde \Sigma^2)^{-1/2},
\qquad\tilde C=\tilde\Sigma\tilde S,
\qquad\tilde X=\tilde W_{\RR}\tilde S.
\]

\subsection{A FEAST algorithm for GSVD}
Suppose that we are interested in computing the generalized singular values
in the interval \((\alpha,\beta)\) with \(\alpha\geq0\).
Let \(\Gamma^+\) be a contour that encloses \((\alpha,\beta)\).
The corresponding spectral projector is
\[
P^+(\check A,\check B)
=\frac{1}{2\pi\mi}\int_{\Gamma^+}(\xi\check B-\check A)^{-1}\check B\md\xi.
\]
For approximate generalized singular vectors \(\tilde U\) and \(\tilde W\),
the filtering scheme
\begin{equation}
\label{eq:proj+GSVD}
\bmat{\tilde U_{\new} \\ \tilde W_{\new}}
=P^+(\check A,\check B)\bmat{\tilde U \\ \tilde W}
\end{equation}
is expected to produce a better approximation.
Similar to what we have discussed in Section~\ref{sec:algorithm-SVD} for SVD,
applying~\eqref{eq:proj+GSVD} is not robust in general.
A safer filtering scheme
\begin{equation}
\label{eq:proj+pair-GSVD}
\bmat{\tilde U_{\new} \\ \tilde W_{\new}}
=P^+(\check A,\check B)\bmat{\tilde U & \tilde U \\ \tilde W & -\tilde W}
=\bmat{P^+(\check A,\check B)\bmat{\tilde U \\ \tilde W},
\bmat{I_m \\ & -I_n}P^-(\check A,\check B)\bmat{\tilde U \\ \tilde W}}
\end{equation}
needs to be applied in the first iteration, where
\[
P^-(\check A,\check B)
=\frac{1}{2\pi\mi}\int_{\Gamma^-}(\xi\check B-\check A)^{-1}\check B\md\xi
=\bmat{I_m \\ & -I_n}P^+(\check A,\check B)\bmat{I_m \\ & -I_n}
\]
is the spectral projector associated with the contour \(\Gamma^-\) that is
symmetric to \(\Gamma^+\) with respect to the origin.
Algorithm~\ref{alg:FEAST-GSVD} summarizes a FEAST algorithm for GSVD based
on~\eqref{eq:proj+GSVD} and~\eqref{eq:proj+pair-GSVD}.

\begin{algorithm}[!tb]
\caption{A FEAST-GSVD algorithm.}
\label{alg:FEAST-GSVD}
\begin{algorithmic}[1]
\REQUIRE Two matrices \(A\in\mathbb{C}^{m\times n}\),
\(B\in\mathbb{C}^{p\times n}\);
location of desired generalized singular values \((\alpha,\beta)\);
initial guess \(U_0\in \mathbb{C}^{m\times \ell}\) and
\(W_0\in \mathbb{C}^{n\times \ell}\) satisfying
\(U_0\herm U_0=W_0\herm B\herm BW_0=I_\ell\);
quadrature nodes with weights \((\xi_j,\omega_j)\)'s for \(j=1\), \(2\),
\(\dotsc\), \(N\).
The number of desired generalized singular values \(k\) cannot exceed \(\ell\).
\ENSURE Approximate generalized singular values \(\Sigma\) and the
corresponding left and right generalized singular vectors \((U,W)\).
\STATE \(U\gets U_0\), \(W\gets W_0\).
\FOR{\({\mathtt{iter}}=1,2,\dotsc\) until convergence}
  \IF{\({\mathtt{iter}}=1\)}
    \STATE \(Z\gets\bmat{U & U \\ W & -W}\).
  \ELSE
    \STATE \(Z\gets\bmat{U \\ W}\).
  \ENDIF
  \STATE \([U\herm,W\herm]\herm\gets\sum_{j=1}^N\omega_j(\xi_j\check B-\check A)^{-1}\check BZ\).
  \STATE Orthogonalize \(U\) and \(W\) so that \(U\herm U=W\herm B\herm BW=I\).
  \STATE Solve SVD of \(A_{\proj}=U\herm AW\) to obtain the singular triplet
    \((\Sigma,U_{\proj},W_{\proj})\).
  \STATE \(U\gets UU_{\proj}\), \(W\gets WW_{\proj}\).
  \STATE Check convergence.
  \IF{\({\mathtt{iter}}=1\)}
    \STATE Choose \(\ell\) best components from \((\Sigma,U,W)\) as the new \((\Sigma,U,W)\).
  \ENDIF
\ENDFOR
\end{algorithmic}
\end{algorithm}

We remark that additional care is required in order to estimate the trace of
the spectral projector for GSVD.
Since
\[
P^+(\check A,\check B)
=\frac{1}{2\pi\mi}\int_{\Gamma^+}(\xi\check B-\check A)^{-1}\check B\md\xi
\]
is not Hermitian, the classical Monte Carlo trace estimator can be very
inaccurate.
To avoid this difficulty, we observe that \(\trace\bigl(P^+(\check A,\check B)\bigr)\) is equal to the trace
of
\begin{equation}
\label{eq:trace-GSVD}
\frac{1}{2\pi\mi}C\herm
\int_{\Gamma^+}(\xi\check B-\check A)^{-1}\md\xi\cdot C
\end{equation}
where \(C=\diag\set{I_m,B\herm}\).
The matrix in~\eqref{eq:trace-GSVD} is Hermitian and positive semidefinite so
that
\[
\trace\biggl(\frac{1}{2\pi\mi}C\herm
\int_{\Gamma^+}(\xi\check B-\check A)^{-1}\md\xi\cdot C\biggr)
=\frac{1}{2\pi\mi}\Expect\biggl[y\herm C\herm
\int_{\Gamma^+}(\xi\check B-\check A)^{-1}Cy\md\xi\biggr]
\]
can be accurately estimated through several independent samples of \(y\) with
\(\Expect[yy\herm]=I_{m+p}\).

\section{Numerical Experiments}
\label{sec:experiments}
In this section we report experimental results of
Algorithms~\ref{alg:FEAST-SVD} and~\ref{alg:FEAST-GSVD}.
All numerical experiments were performed using MATLAB R2022b on a Linux server
with two 16-core Intel Xeon Gold~6226R~2.90~GHz CPUs and 1024~GB of main
memory.

\subsection{Experiment settings}
\label{subsec:setting}
We choose twelve test matrices from the SuiteSparse Matrix Collection
(formally, the University of Florida Sparse Matrix Collection)~\cite{DH2011},
as listed in Table~\ref{tab:matrices}.
The \texttt{rosen10}\(\trans\) and \texttt{flower\_5\_4}\(\trans\),
respectively, represent the transpose of \texttt{rosen10} and
\texttt{flower\_5\_4}.
For GSVD, the matrix \(B\) is taken as the scaled transpose of the discretized
first order derivative, i.e.,
\[
B=\bmat{1 & -1 \\ & 1 & -1 \\ & & \ddots & \ddots \\ & & & 1 & -1}\trans
\in\mathbb C^{(n+1)\times n}.
\]
Table~\ref{tab:matrices} also contains other information regarding the
experiments, such as the desired interval \((\alpha,\beta)\), the number of
(generalized) singular values in the interval \(\kSVD\) (or \(\kGSVD\)),
reported by \texttt{eig()} in MATLAB, as well as the estimated number
\(\tildekSVD\) (or \(\tildekGSVD\)) through stochastic trace estimator.
These test matrices are chosen from a broad range of applications, and have
been used by researchers for testing SVD/GSVD algorithms~\cite{JZ2022}.

\begin{table}[tb!]
\setlength{\tabcolsep}{3pt}
\centering
\caption{List of test matrices.}
\label{tab:matrices}
\begin{tabular}{ccccccccccc}
\hline
ID & Matrix \(A\) & \(m\) & \(n\) & \(\nnz(A)\) & \(\lVert A\rVert_2\) &
\((\alpha,\beta)\) &\(\kSVD\)  &\(\tildekSVD\) &\(\kGSVD\) &\(\tildekGSVD\)\\
\hline
 1 & \texttt{plat1919}               & \phantom{00}1,919 & \phantom{0}1,919 & \phantom{0}32,399
                                     & \phantom{00}2.93  & \((2.1,2.5) \) & 8   & 8.42   & 5  & 4.62 \\
 2 & \texttt{rosen10}\(\trans\)      & \phantom{00}6,152 & \phantom{0}2,056 & \phantom{0}68,302
                                     & \phantom{}20,200  & \((9.9,10.1)\) & 12  & 11.32  & 17 & 16.47\\
 3 & \texttt{GL7d12}                 & \phantom{00}8,899 & \phantom{0}1,019 & \phantom{0}37,519
                                     & \phantom{00}14.4  & \((11,12)   \) & 17  & 17.91  & 19 & 18.23\\
 4 & \texttt{3elt\_dual}             & \phantom{00}9,000 & \phantom{0}9,000 & \phantom{0}26,556
                                     & \phantom{0000}3   & \((1.5,1.6) \) & 368 & 346.31 & 251& 261.46\\
 5 & \texttt{fv1}                    & \phantom{00}9,604 & \phantom{0}9,604 & \phantom{0}85,264
                                     & \phantom{00}4.52  & \((3.1,3.15)\) & 89  & 84.18  &56  & 56.21\\
 6 & \texttt{shuttle\_eddy}          & \phantom{0}10,429 & \phantom{}10,429 & \phantom{}103,599
                                     & \phantom{00}16.2  & \((7,7.01)  \) & 6   & 7.24   & 2  & 2.07\\
 7 & \texttt{nopoly}                 & \phantom{0}10,774 & \phantom{}10,774 & \phantom{0}70,842
                                     & \phantom{00}23.3  & \((12,12.5) \) & 340 & 352.44 & 159& 160.45\\
 8 & \texttt{flower\_5\_4}\(\trans\) & \phantom{0}14,721  & \phantom{0}5,226& \phantom{0}43,942
                                     & \phantom{00}5.53  & \((4.1,4.3) \) & 137 & 136.75 & 51 & 53.09\\
 9 & \texttt{barth5}                 & \phantom{0}15,606 & \phantom{}15,606 & \phantom{0}61,484
                                     & \phantom{00}4.23  & \((1.5,1.6) \) & 384 & 389.89 & 317& 326.99\\
10 & \texttt{L-9}                    & \phantom{0}17,983 & \phantom{}17,983 & \phantom{0}71,192
                                     & \phantom{0000}4   & \((1.2,1.3) \) & 477 & 489.40& 607 & 619.06\\
11 & \texttt{crack\_dual}            & \phantom{0}20,141 & \phantom{}20,141 & \phantom{0}60,086
                                     & \phantom{0000}3   & \((1,1.1)   \) & 330 & 329.34 & 602 & 601.41\\
12 & \texttt{rel8}                   & \phantom{}345,688 & \phantom{}12,347 & \phantom{}821,839
                                     & \phantom{00}18.3  & \((13,14)   \) & 13  & 12.69  & 185 & 186.77\\
\hline
\end{tabular}
\end{table}

To check the correctness of computed GSVD components
\((\hat\alpha_i,\hat\beta_i,\hat u_i,\hat v_i,\hat x_i)\)'s, we need to verify
\begin{subequations}
\label{eq:conv-orig}
\begin{align}
\lVert A\hat x_i-\hat u_i\hat\alpha_i\rVert_2
&=O(\macheps)\cdot\bigl(\lVert A\rVert_2\lVert\hat x_i\rVert_2
+\hat\alpha_i\bigr),\label{eq:conv-orig-1}\\
\lVert B\hat x_i-\hat v_i\hat\beta_i\rVert_2
&=O(\macheps)\cdot\bigl(\lVert B\rVert_2\lVert\hat x_i\rVert_2
+\hat\beta_i\bigr),\label{eq:conv-orig-2}\\
\lVert A\herm\hat u_i\hat\beta_i-B\herm\hat v_i\hat\alpha_i\rVert_2
&=O(\macheps)\cdot\bigl(\hat\beta_i\lVert A\rVert_2
+\hat\alpha_i\lVert B\rVert_2\bigr),\label{eq:conv-orig-3}\\
\lVert\hat U\herm\hat U-I\rVert_2&=O(\macheps),
\label{eq:conv-orig-4}\\
\lVert\hat V\herm\hat V-I\rVert_2&=O(\macheps).
\label{eq:conv-orig-5}
\end{align}
\end{subequations}
where \(\macheps\) is the unit roundoff,
and \(\hat U=[\hat u_1,\hat u_2,\dotsc]\),
\(\hat V=[\hat v_1,\hat v_2,\dotsc]\).
Equations~\eqref{eq:conv-orig-2}, \eqref{eq:conv-orig-4},
and~\eqref{eq:conv-orig-5} are automatically ensured in
Algorithm~\ref{alg:FEAST-GSVD}.
As \(w_i\)'s instead of \(x_i\)'s are computed in
Algorithm~\ref{alg:FEAST-GSVD}, in practice we use the convergence criteria
\begin{subequations}
\label{eq:conv}
\begin{align}
\lVert A\hat w_i-\hat u_i\hat\sigma_i\rVert_2
&\leq\tol\cdot\bigl(\lVert A\rVert_2\lVert\hat w_i\rVert_2
+\lvert\hat\sigma_i\rvert\bigr),\label{eq:conv-1}\\
\lVert A\herm\hat u_i-B\herm B\hat w_i\hat\sigma_i\rVert_2
&\leq\tol\cdot\bigl(\lVert A\rVert_2
+\lvert\hat\sigma_i\rvert\lVert B\rVert_2^2\lVert\hat w_i\rVert_2\bigr),
\label{eq:conv-2}
\end{align}
\end{subequations}
where \(\tol\) is a user-specified threshold.
However, we remark that as an interior eigensolver, the number of Ritz values
within the desired interval \((\alpha,\beta)\) can be larger than the actual
number of generalized singular values in \((\alpha,\beta)\).
Thus it is impossible to request all of these Ritz values to converge.
As a remedy, we impose an additional stopping criterion---the algorithm
terminates when either all Ritz values within \((\alpha,\beta)\) have converged
according to~\eqref{eq:conv}, or the number of converged Ritz values remains
unchanged in two consecutive iterations.

In all runs, the initial guess is randomly generated unless otherwise
specified, with \(\ell=\lceil3/2\cdot\tildekSVD\rceil+5\) (or
\(\ell=\lceil3/2\cdot\tildekGSVD\rceil+5\)) columns, where \(\tildekSVD\) (or
\(\tildekGSVD\)) is estimated using \(12\) quadrature nodes and \(K=30\) random
trial vectors.
The contour \(\Gamma^+\) is chosen as the ellipse with major axis
\([\alpha,\beta]\) and aspect ratio \(\rho=a/b=5\).
The number of quadrature nodes on \(\Gamma^+\) is set to \(N=12\).
The convergence threshold is set to \(\tol=10^{-14}\sqrt m\).

\subsection{Convergence history}
Figure~\ref{fig:performance_svd} shows the convergence history for computing
the singular values in the desired interval \((\alpha,\beta)\).
We plot the maximum relative residual
\[
\max_{\sigma_i\in(\alpha,\beta)}\max\biggl\lbrace
\frac{\lVert A\hat w_i-\hat u_i\hat\sigma_i\rVert_2}
{\lVert A\rVert_2\lVert\hat w_i\rVert_2+\lvert\hat\sigma_i\rvert},
\frac{\lVert A\herm\hat u_i-B\herm B\hat w_i\hat\sigma_i\rVert_2}
{\lVert A\rVert_2
+\lvert\hat\sigma_i\rvert\lVert B\rVert_2^2\lVert\hat w_i\rVert_2}
\biggr\rbrace
\]
for singular values inside \((\alpha,\beta)\) in all convergence plots
throughout this section.
Algorithm~\ref{alg:FEAST-SVD} typically finds all singular values in the
desired interval within three iterations for most matrices, and, moreover, due
to our additional convergence criterion the third iteration is merely for
confirming the convergence.
Even for matrices with relatively slow convergence, most singular values are
found in the first two or three iterations.
Subsequent iterations only make the marginal contribution on the convergence of
one or two singular values.

\begin{figure}[tb!]
\centering
\begin{tabular}{ccc}
\includegraphics[width=0.3\textwidth]{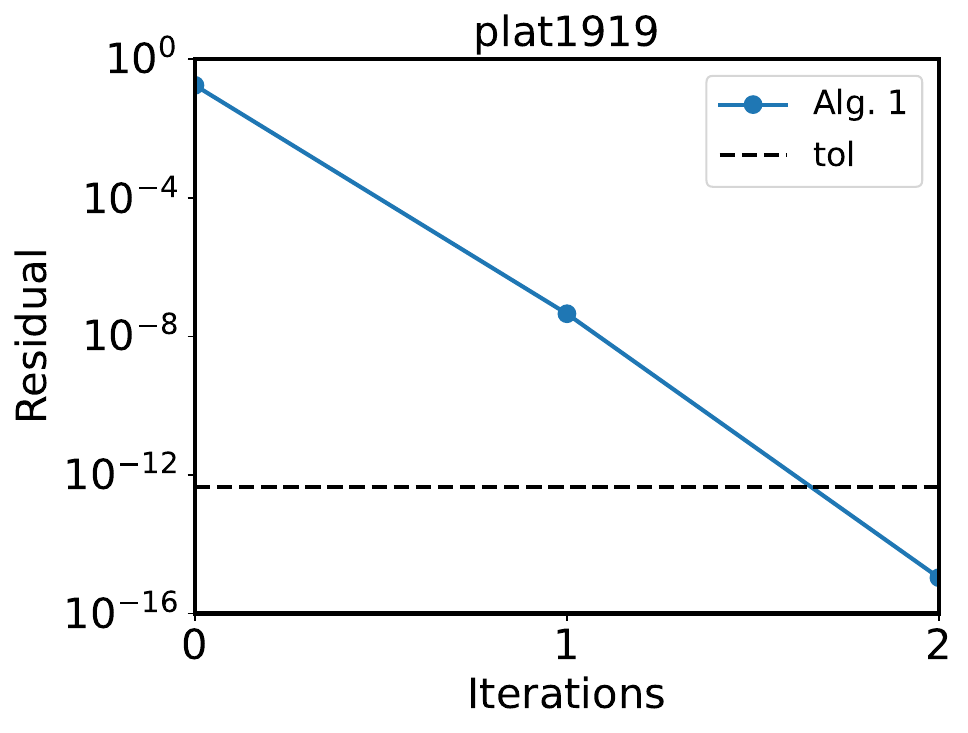} &
\includegraphics[width=0.3\textwidth]{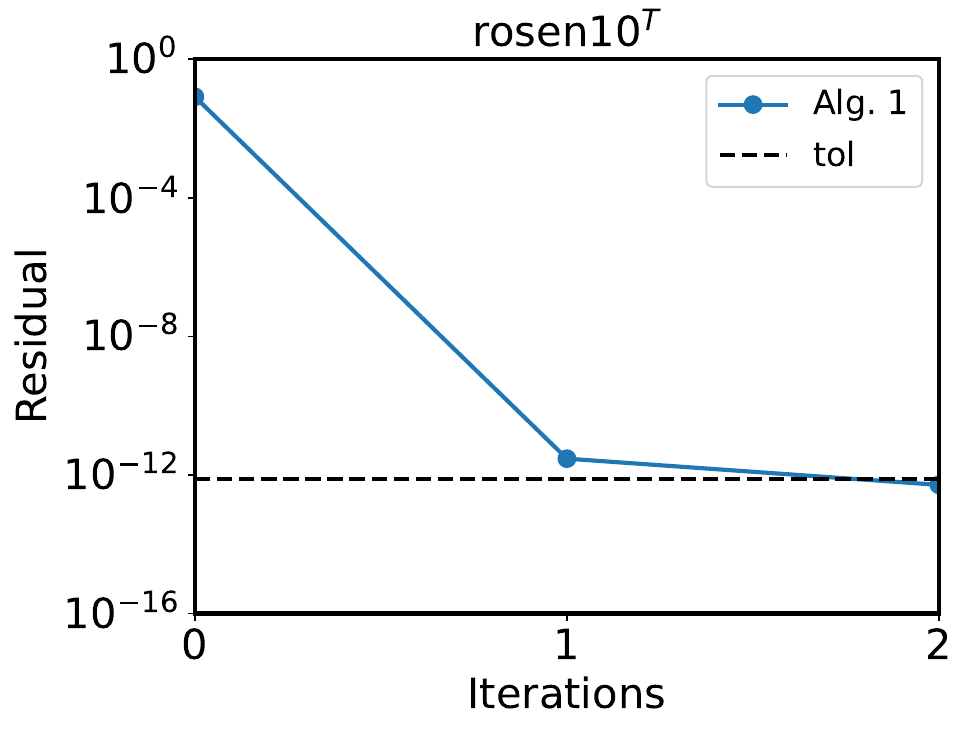} &
\includegraphics[width=0.3\textwidth]{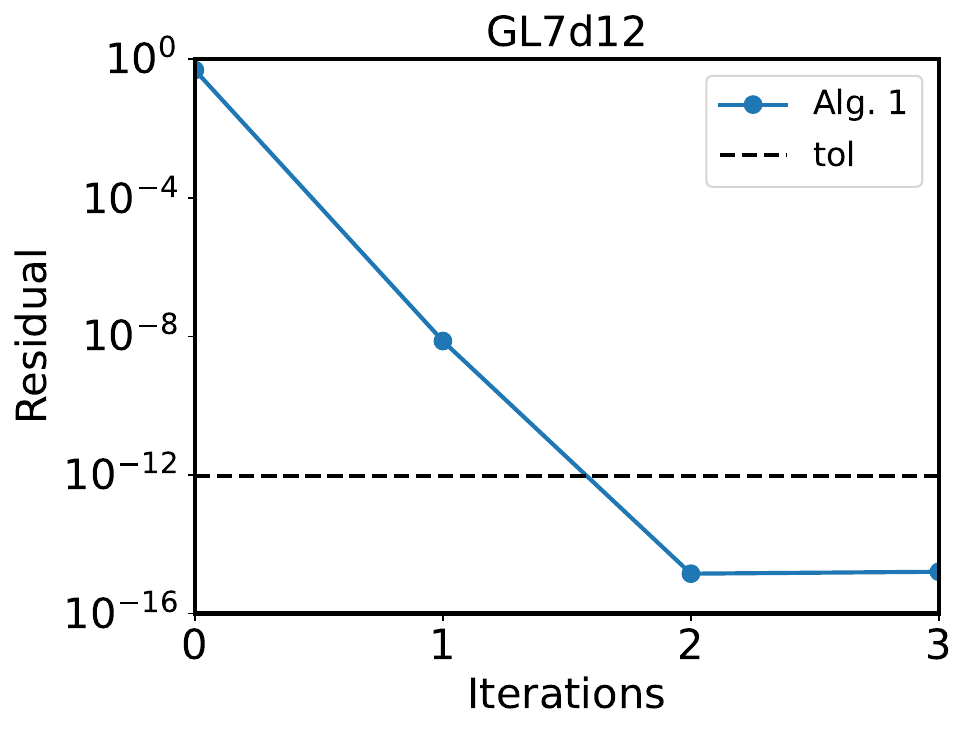} \\
\includegraphics[width=0.3\textwidth]{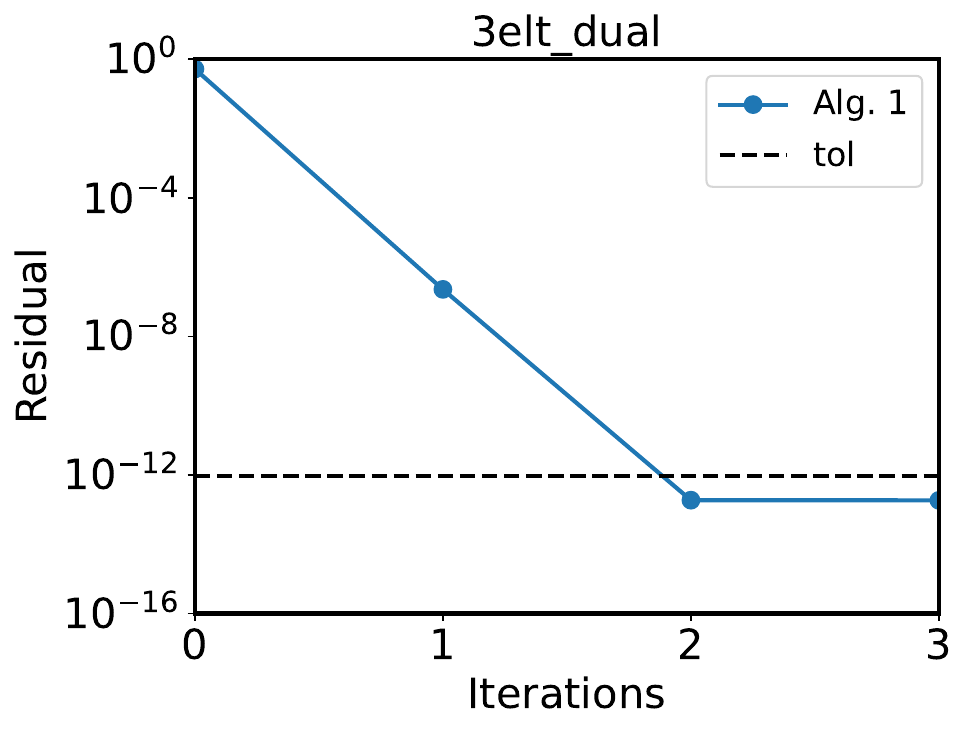} &
\includegraphics[width=0.3\textwidth]{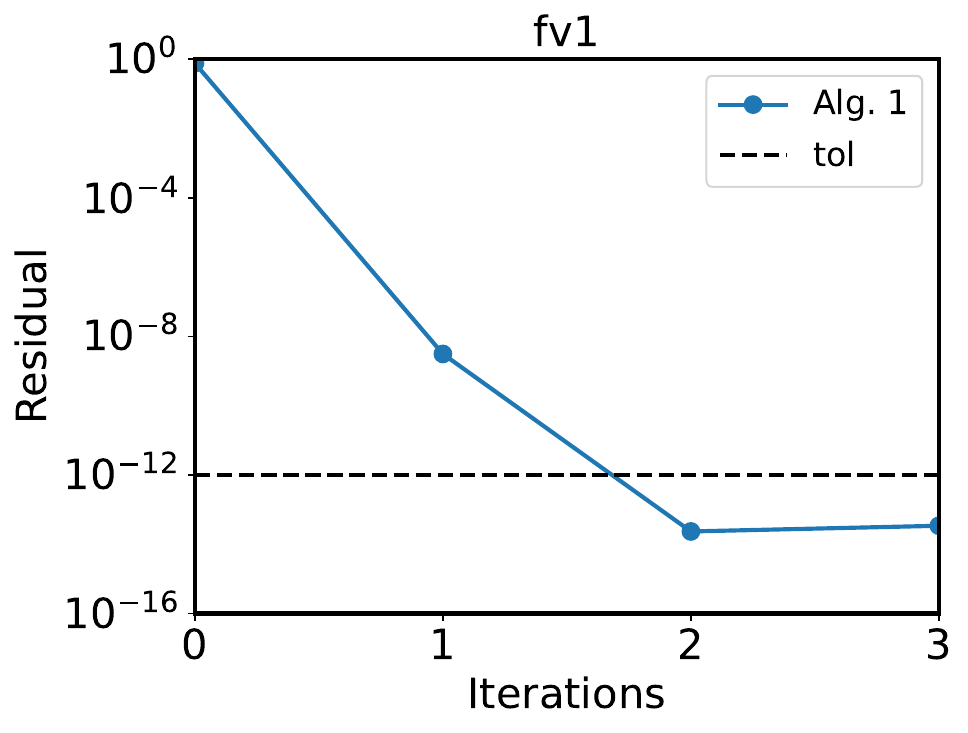} &
\includegraphics[width=0.3\textwidth]{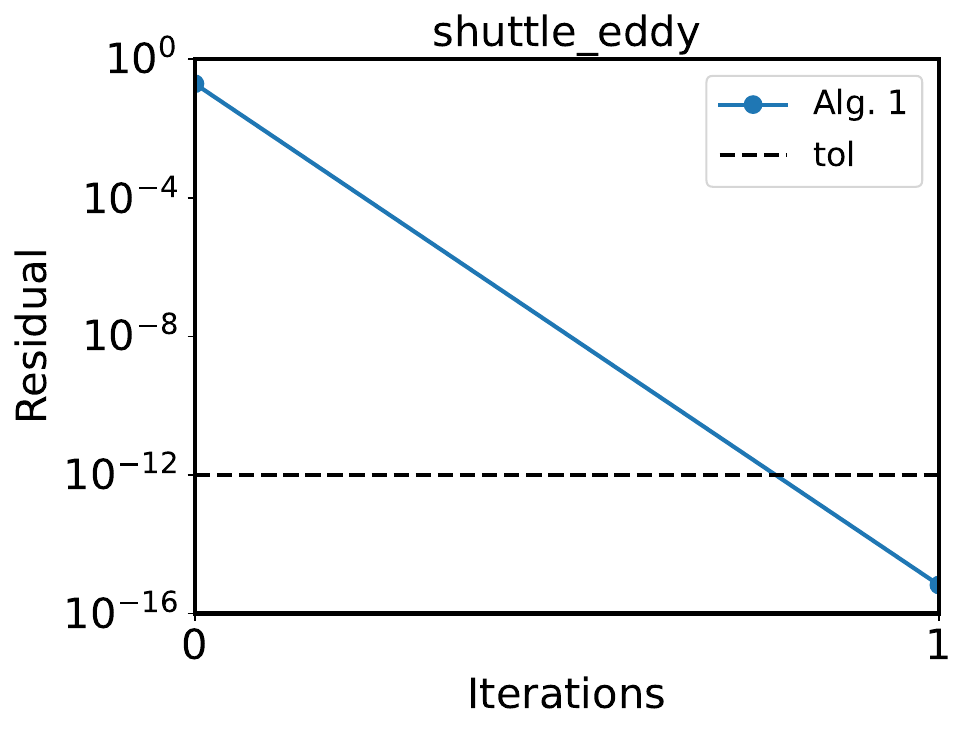} \\
\includegraphics[width=0.3\textwidth]{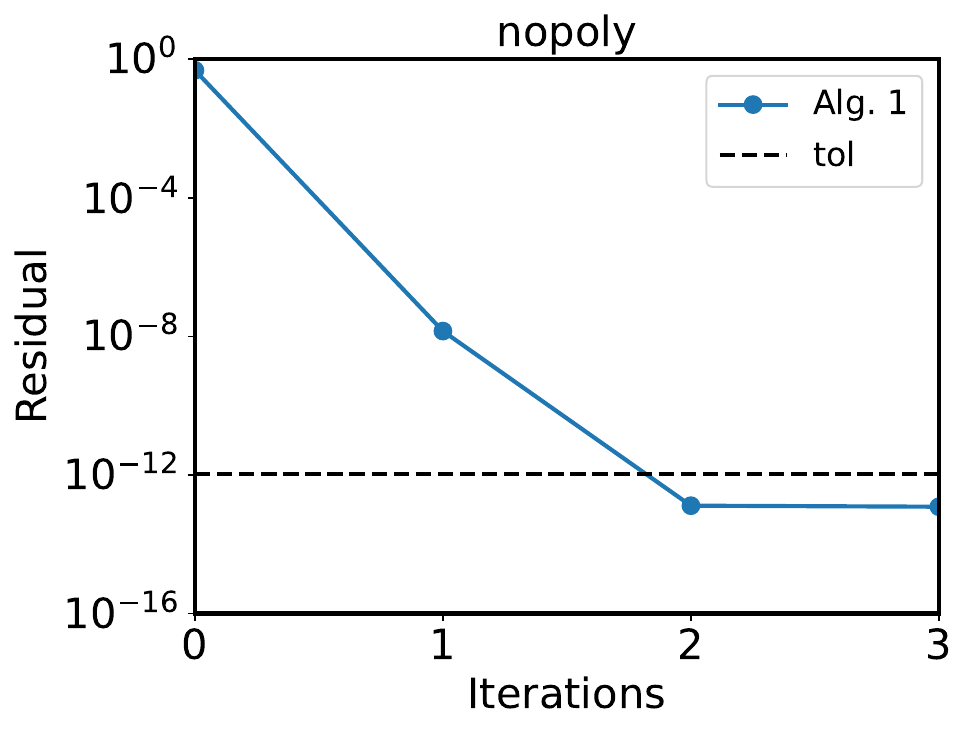} &
\includegraphics[width=0.3\textwidth]{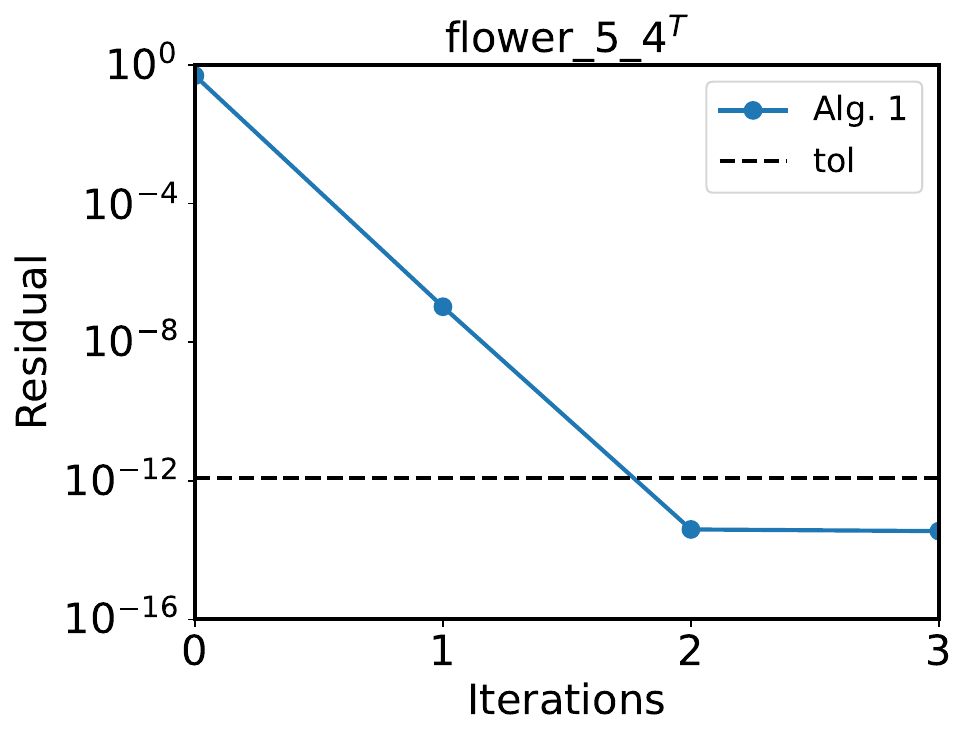} &
\includegraphics[width=0.3\textwidth]{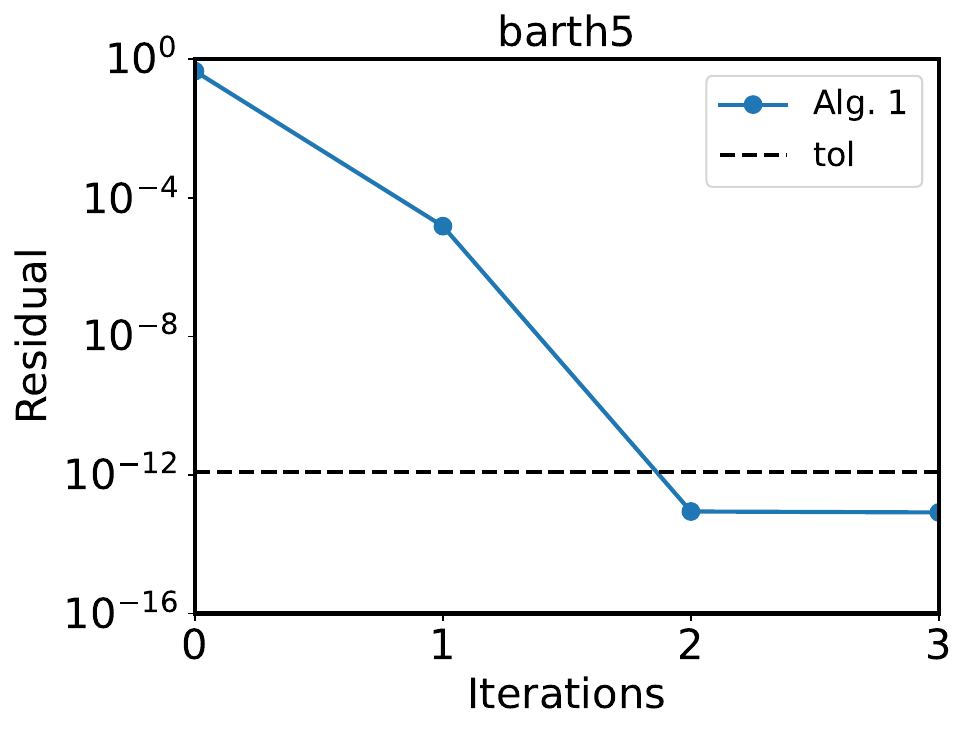} \\
\includegraphics[width=0.3\textwidth]{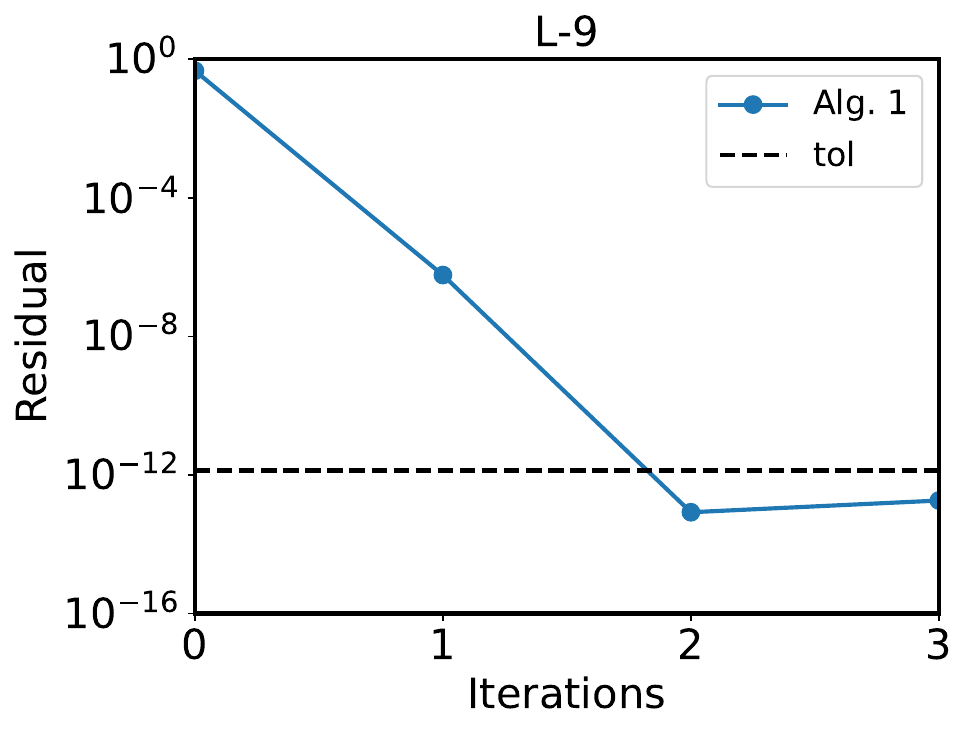} &
\includegraphics[width=0.3\textwidth]{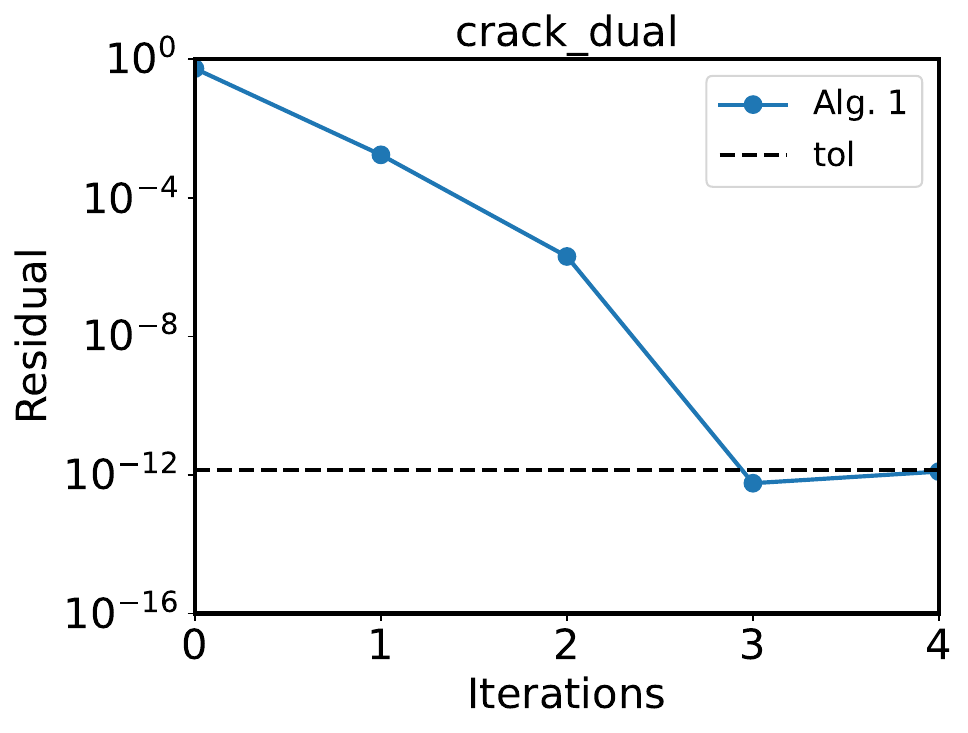} &
\includegraphics[width=0.3\textwidth]{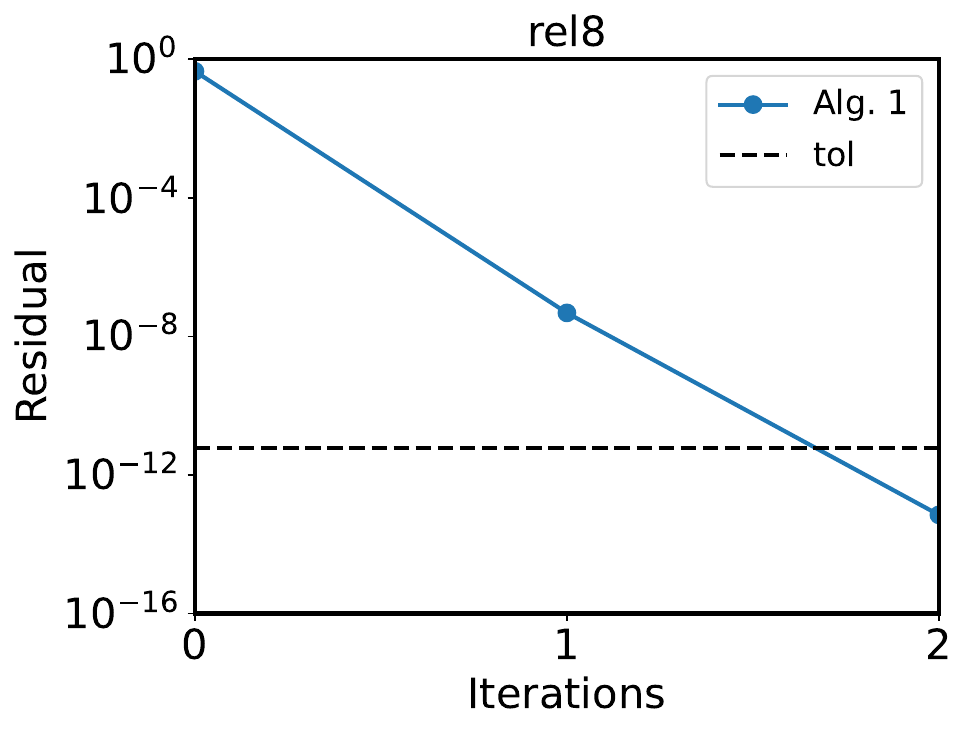}
\end{tabular}
\caption{Convergence history for SVD experiments.}
\label{fig:performance_svd}
\end{figure}

For GSVD, the behavior of Algorithm~\ref{alg:FEAST-GSVD} is similar.
Figure~\ref{fig:performance_gsvd} shows the convergence history.
Algorithm~\ref{alg:FEAST-GSVD} successfully finds all generalized singular
values in the desired interval within three or four iterations for our
examples.

\begin{figure}[tb!]
\centering
\begin{tabular}{ccc}
\includegraphics[width=0.3\textwidth]{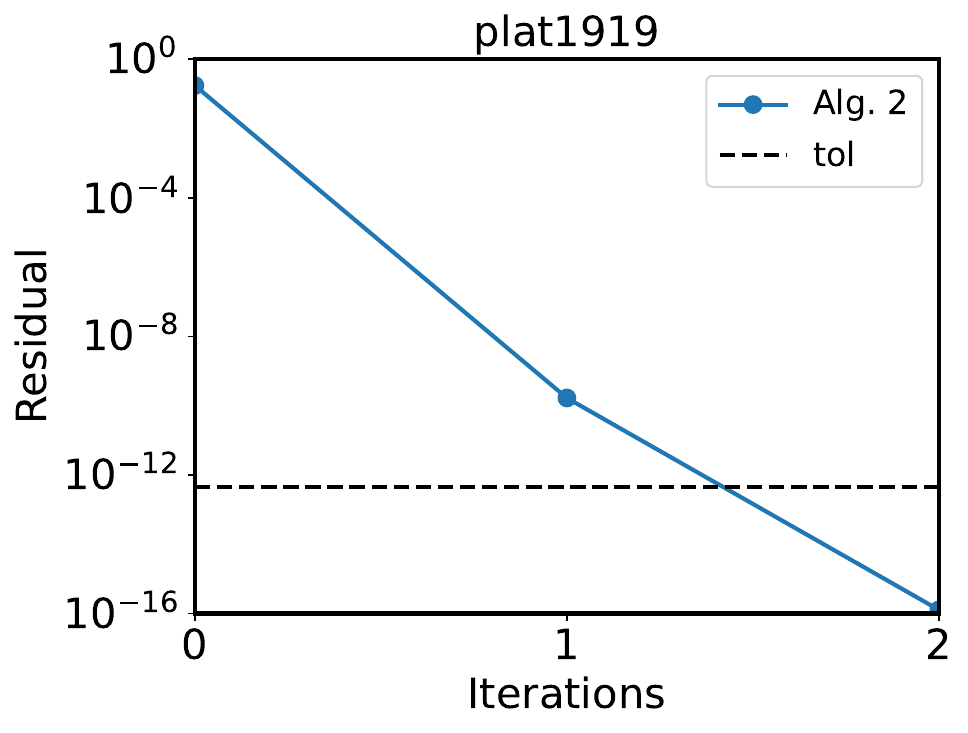} &
\includegraphics[width=0.3\textwidth]{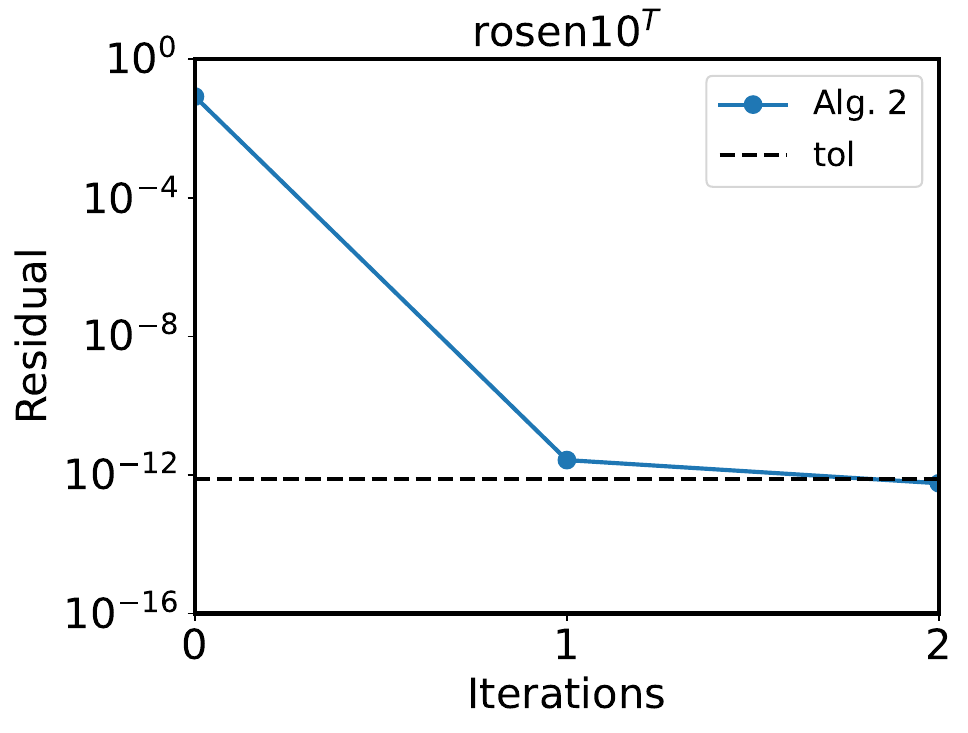} &
\includegraphics[width=0.3\textwidth]{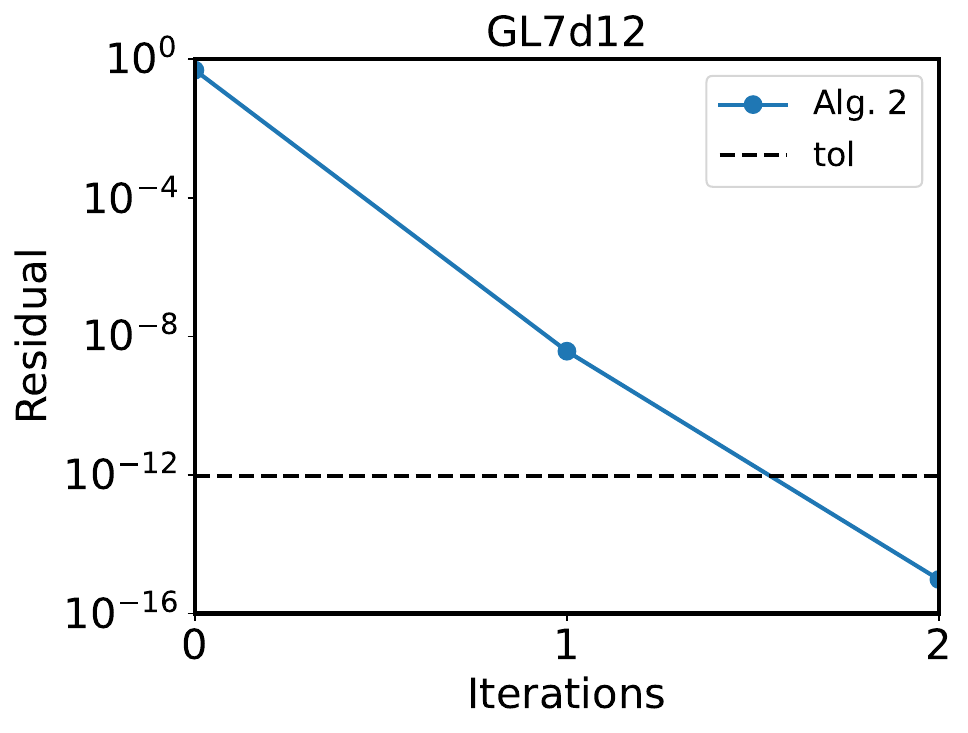} \\
\includegraphics[width=0.3\textwidth]{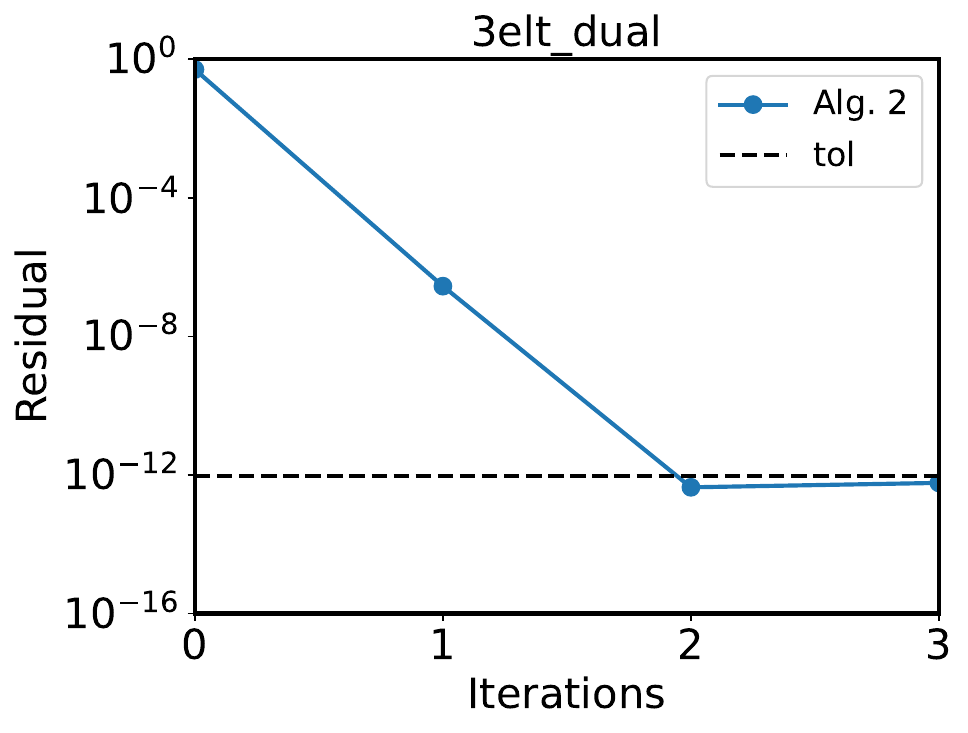} &
\includegraphics[width=0.3\textwidth]{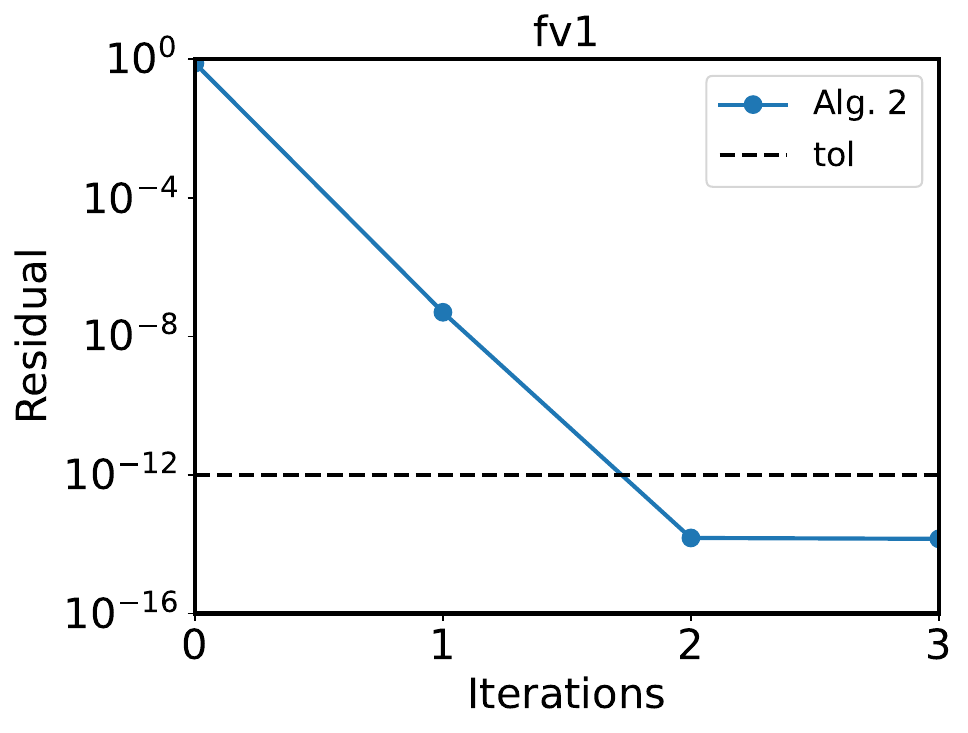} &
\includegraphics[width=0.3\textwidth]{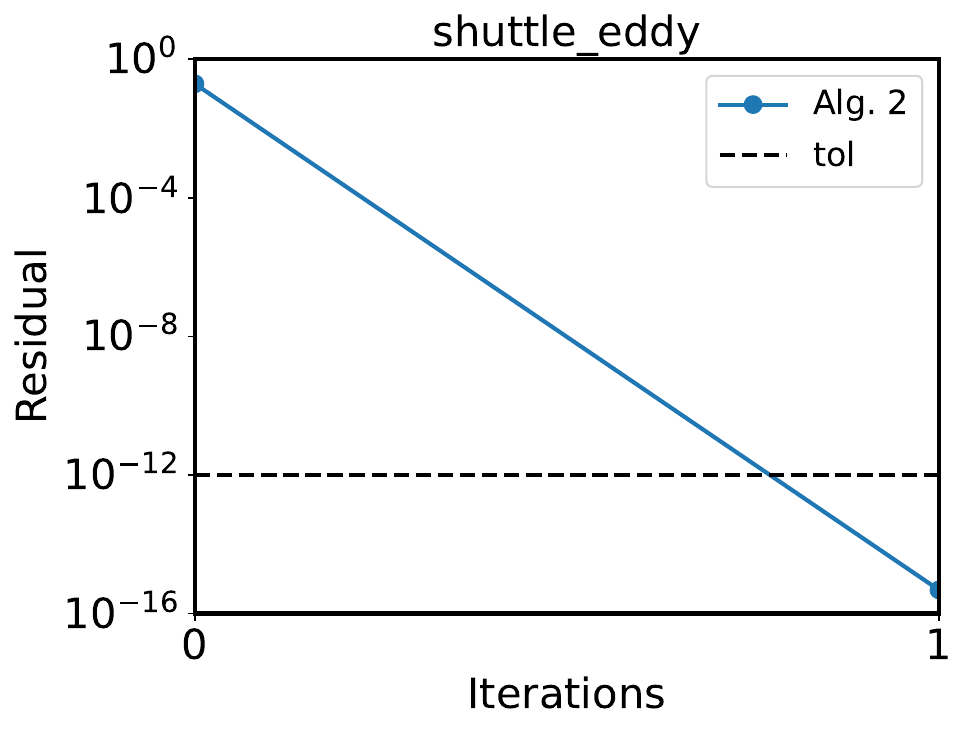} \\
\includegraphics[width=0.3\textwidth]{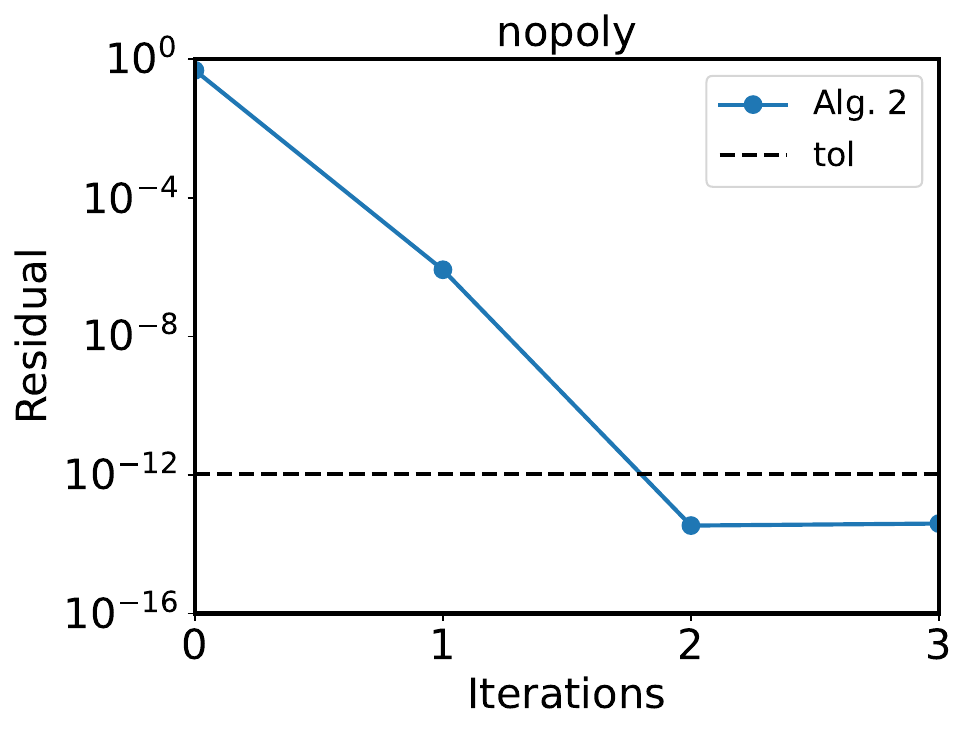} &
\includegraphics[width=0.3\textwidth]{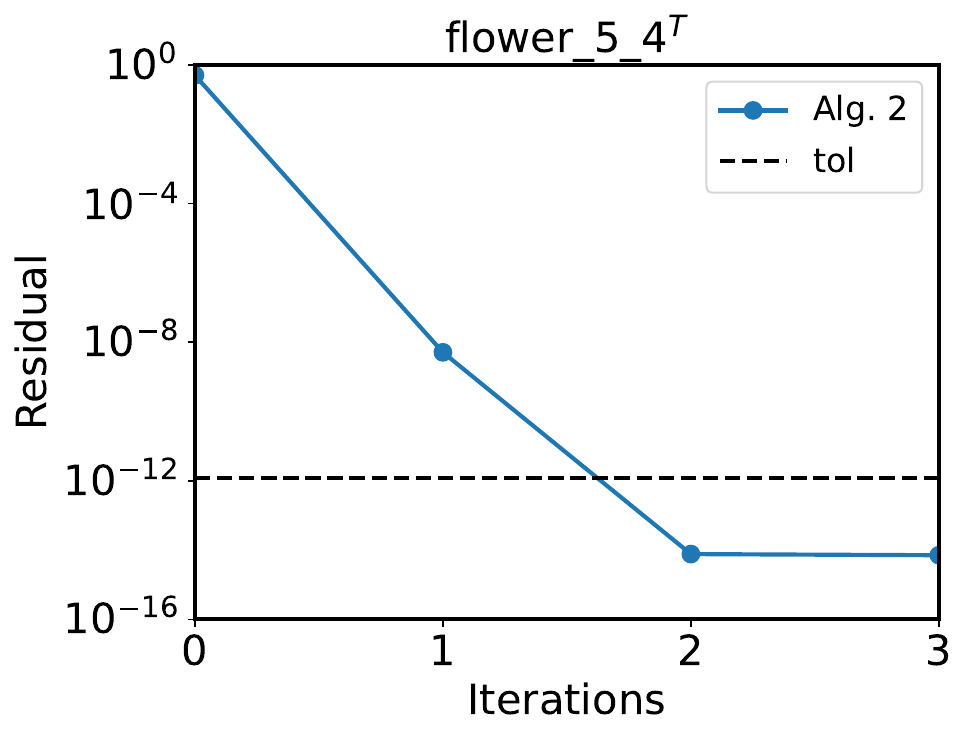} &
\includegraphics[width=0.3\textwidth]{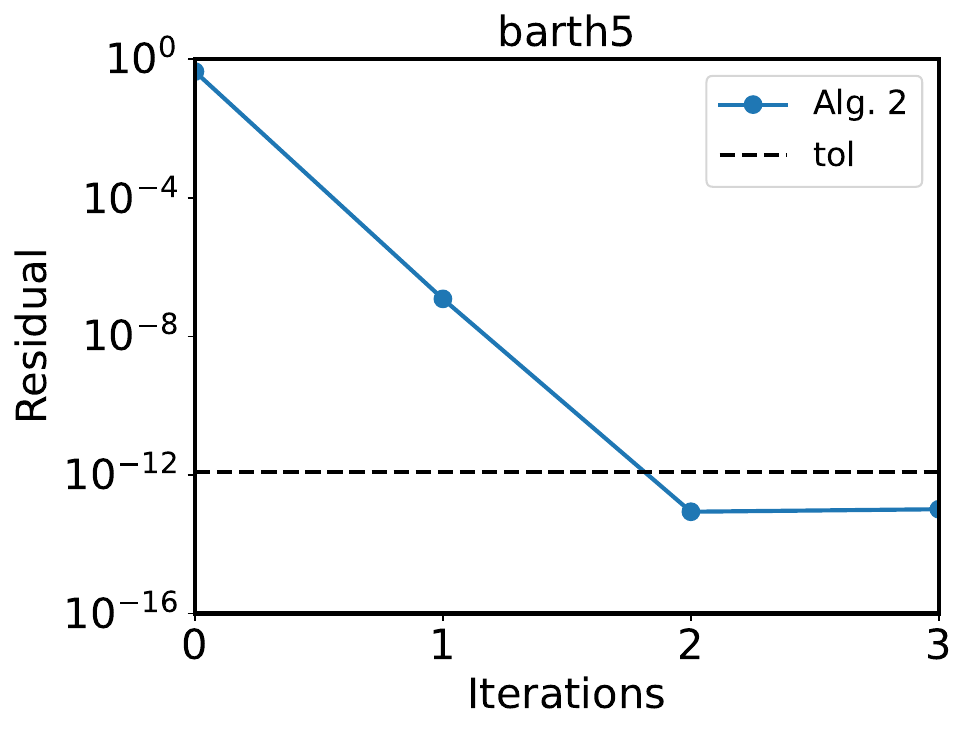} \\
\includegraphics[width=0.3\textwidth]{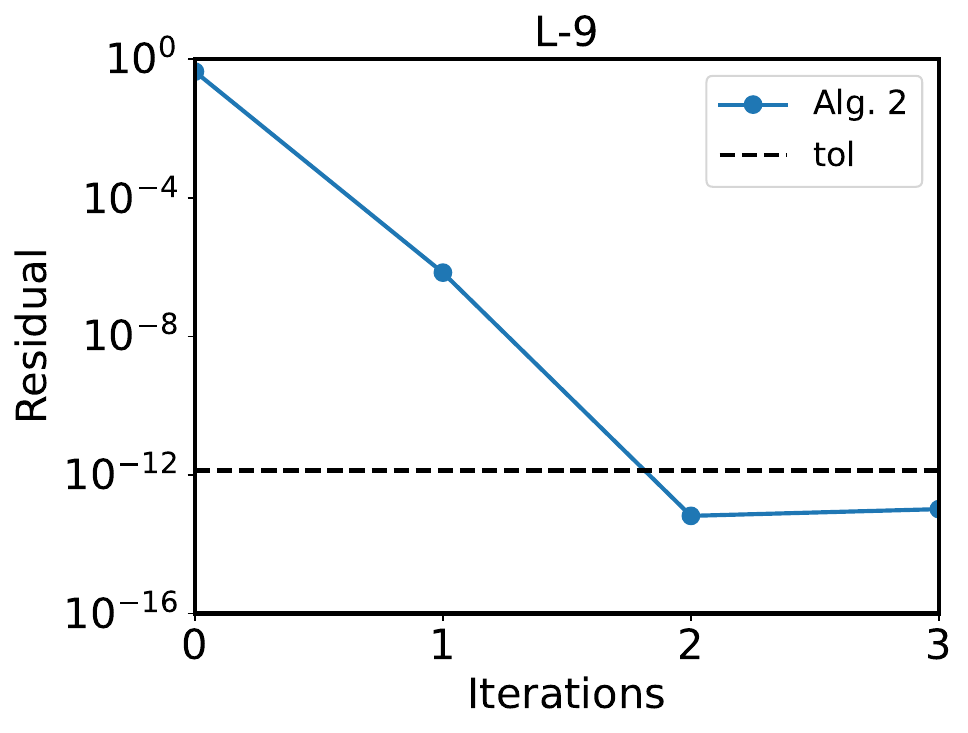} &
\includegraphics[width=0.3\textwidth]{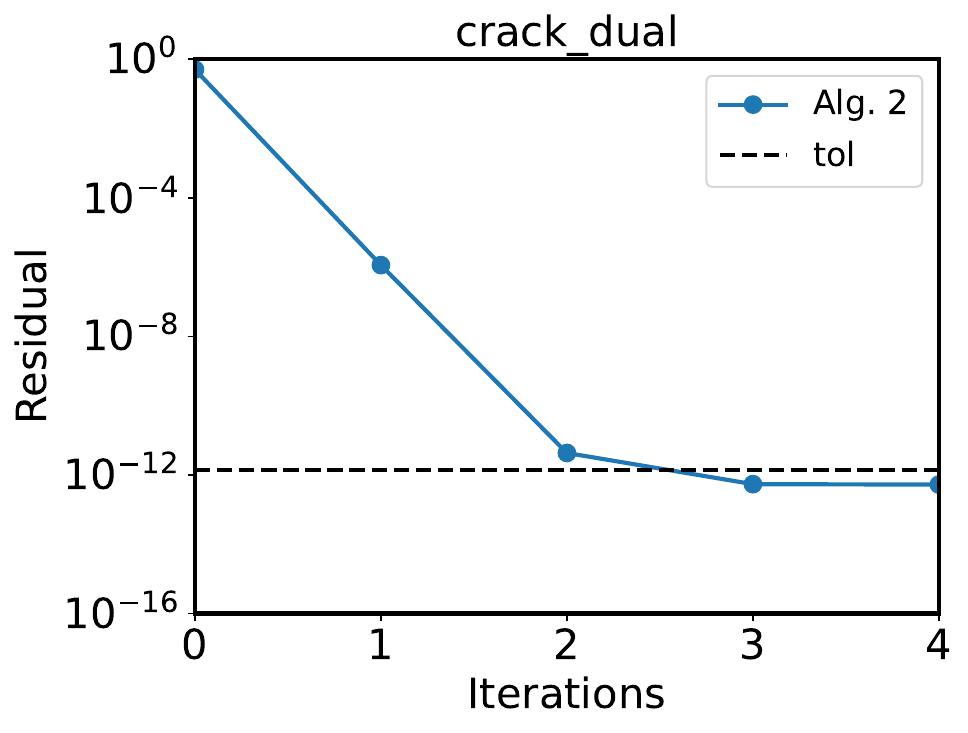} &
\includegraphics[width=0.3\textwidth]{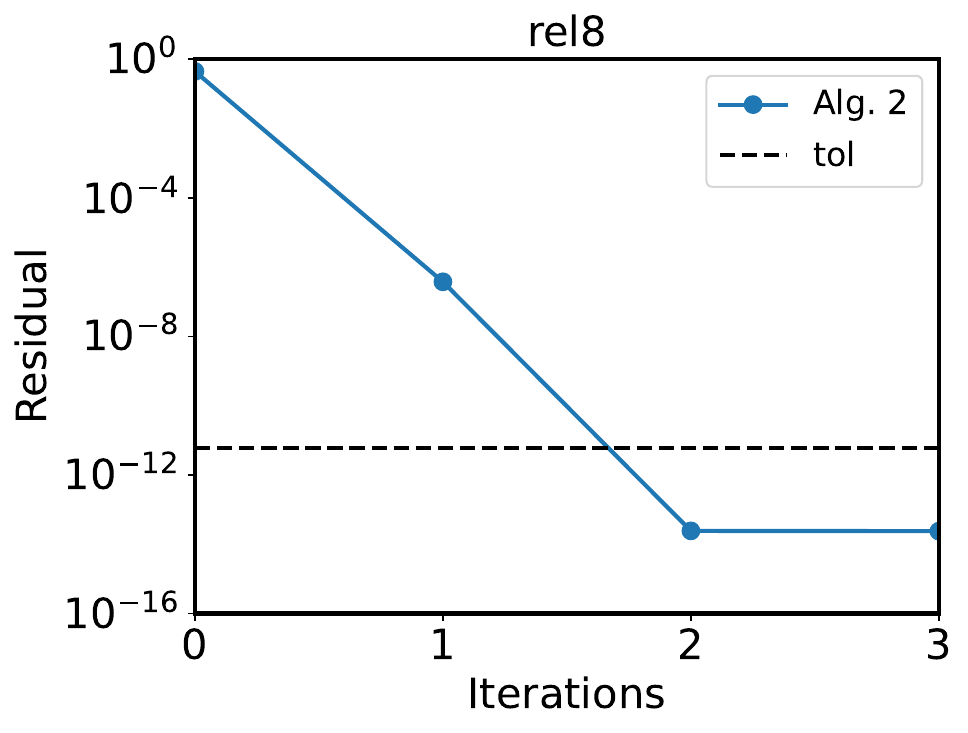}
\end{tabular}
\caption{Convergence history for GSVD experiments.}
\label{fig:performance_gsvd}
\end{figure}

\subsection{Comparison with Jacobi--Davidson and subspace iteration}
To illustrate the effectiveness and efficiency of our algorithm,
we test our examples with two frequently used general-purpose interior
eigensolvers---a cross-product free Jacobi--Davidson (JD)~\cite{HJ2023} and
the shift-and-invert-based subspace iteration (SI) algorithm.
All algorithms use the same randomly generated initial guess from an
\(\ell\)-dimensional subspace.
For the Jacobi--Davidson method, we set \(k_{\min}=5\), \(k_{\max}=50\),
\({\rm fixtol}=10^{-4}\), and \(\tilde\epsilon=10^{-3}\);
see~\cite{HJ2023} for a detailed explanation of these parameters.
The inner iteration for solving the correction equation is solved by GMRES with
the shift-and-invert preconditioner, and the maximum iteration count is set to
\(20\).
For subspace iteration, \(\tau=(\alpha+\beta)/2\) is chosen as the shift.
More precisely, we apply subspace iteration to the matrix
\((\check A-\tau\check B)^{-1}\).

In Tables~\ref{tab:conv_num_SVD} and~\ref{tab:conv_num_GSVD}, we list the
number of converged (generalized) singular values as well as the execution
time for JD and SI.
Execution time of Algorithm~\ref{alg:FEAST-SVD} or~\ref{alg:FEAST-GSVD} is
also listed for reference.
To avoid abuse of computational resources, we have set the time limit for each
test to half an hour.
Within this time limit, JD and SI only solve relatively small problems,
especially in the GSVD setting.
For all test cases, Algorithm~\ref{alg:FEAST-SVD} or~\ref{alg:FEAST-GSVD} is
much faster compared to JD and SI, mainly because
Algorithm~\ref{alg:FEAST-SVD} or~\ref{alg:FEAST-GSVD} converges very rapidly.

\begin{table}[tb!]
\centering
\caption{The number of converged singular values computed by the
cross-product free Jacobi--Davidson algorithm and the shift-and-invert-based subspace 
iteration algorithm, together with the corresponding execution time (sec).
Numbers in boldface mean that all singular values have
converged within the time limit---half an hour.}
\label{tab:conv_num_SVD}
\resizebox{\textwidth}{!}{
\begin{NiceTabular}{cccccccc}
\hline
\multicolumn{2}{c}{ID}   & 1  &  2 & 3  & 4   & 5  & 6\\
\hline
\multirow{2}*{Alg.~\ref{alg:FEAST-SVD}}  & desired & 8  & 12 & 17 & 368 & 89 & 6\\
                  &time        & 1.54 & 1.05 & 6.14 & 12.64 & 9.63 & 2.16 \\
\hdottedline
\multirow{3}*{JD} &converged   & \textbf{8} & \textbf{12} & \textbf{17}
                               &\textbf{368}& \textbf{89} & \textbf{6}\\
                  &iterations  & 36    & 58   & 93   & 4248 & 528   & 25\\
                  &time        & 2.49  & 1.43 & 8.19 & 540  & 114.3 & 4.64\\
\hdottedline
\multirow{3}*{SI} &converged   & \textbf{8} & \textbf{12} & \textbf{17}
                               & 324 & \textbf{89} & \textbf{6}\\
                  &iterations  & 29    & 34   & 77  & 529 & 90  & 9\\
                  &time        & 3.39  & 3.88 & 26.09 & \(1800^{+}\)  & 171.8 & 3.55\\         
\hline
\hline
\multicolumn{2}{c}{ID}    & 7    & 8   &  9   & 10   & 11  & 12\\
\hline
\multirow{2}*{Alg.~\ref{alg:FEAST-SVD}}  &desired & 340 & 137 & 384 & 477 & 330 & 13\\
                 &time        & 18.28 & 47.32 & 22.06 
                              & 35.04 & 31.46 & 659.6\\
\hdottedline
\multirow{3}*{JD}&converged   & \textbf{340} & \textbf{137} & \textbf{384}
                              & 429 &\textbf{330} & 5\\
                 &iterations  & 3326  & 1120 & 4022 & 5120 & 3871 & 25\\
                 &time        & 634.9 & 1647 & 997  & \(1800^{+}\)
                              & 1301 & \(1800^{+}\)\\
\hdottedline
\multirow{3}*{SI}&converged   & 155  & \textbf{137} & 169  & 225  & 144 & 11\\
                 &iterations  & 207  & 159  & 168   & 115  & 168  & 43\\
                 &time        & \(1800^{+}\) & 1245 & \(1800^{+}\)
                              & \(1800^{+}\) & \(1800^{+}\) & \(1800^{+}\)\\
\hline
\end{NiceTabular}}
\end{table}

\begin{table}[tb!]
\centering
\caption{The number of converged generalized singular values computed by the
cross-product free Jacobi--Davidson algorithm and the shift-and-invert-based subspace 
iteration algorithm, together with the corresponding execution time (sec).
Numbers in boldface mean that all generalized singular values have
converged within the time limit---half an hour.}
\label{tab:conv_num_GSVD}
\resizebox{\textwidth}{!}{
\begin{NiceTabular}{cccccccc}
\hline
\multicolumn{2}{c}{ID}            & 1  &  2 & 3  & 4   & 5  & 6\\
\hline
\multirow{2}*{Alg.~\ref{alg:FEAST-GSVD}} &desired & 5     & 17 & 19  & 251 & 56 & 2 \\
                         &time    & 1.66  & 1.31  & 3.89 & 17.48 & 9.06  & 2.42\\
\hdottedline
\multirow{3}*{JD}        &converged  & \textbf{5}    & \textbf{17} & \textbf{19}
                                     & \textbf{251}  & \textbf{56} & \textbf{2}\\
                         &iterations & 37  & 102  & 140   &3308  &574   & 15\\
                         &time       &2.89 & 3.07 & 15.58 & 1256 &202.5 &4.97\\
\hdottedline
\multirow{3}*{SI}       &converged   & \textbf{5} & \textbf{17} & \textbf{19}
                                     & 214        & \textbf{56}    & \textbf{2}\\
                        &iterations  & 38  & 45   & 39  &488  &180 & 10\\
                        &time        &3.47 & 8.11 & 15.57 & \(1800^{+}\) &215.1 &2.76\\
\hline
\hline
\multicolumn{2}{c}{ID}    & 7    & 8   &  9   & 10   & 11  & 12\\
\hline
\multirow{2}*{Alg.~\ref{alg:FEAST-GSVD}}  &desired  & 159  & 51   & 317  & 607 & 602 & 185\\
                       &time         & 19.31  & 43.24 & 39.68
                                     & 58.69  & 78.41 & 1283\\
\hdottedline
\multirow{3}*{JD}      &converged     & \textbf{159} &\textbf{51} & 200 & 155 & 192 & 4\\
                       &iterations    & 1959 &483 & 2457 & 2032 &2379  & 23\\
                       &time          & 706.3 & 906.6 & \(1800^{+}\)
                                      & \(1800^{+}\) & \(1800^{+}\) & \(1800^{+}\)\\
\hdottedline
\multirow{3}*{SI}      &converged     & 104  &\textbf{51}& 142 & 201 & 218 & 2\\
                       &iterations    & 386 & 251  & 141 & 67  &60   & 5\\
                       &time          & \(1800^{+}\) & 977.5   & \(1800^{+}\)
                                      & \(1800^{+}\) & \(1800^{+}\) & \(1800^{+}\)\\
\hline
\end{NiceTabular}}
\end{table}

\subsection{Comparison on spectral projectors}
In the following we compare several choices of spectral projectors discussed
in Section~\ref{subsec:SVD-projector}.
We apply four different filters, \(P^+\), \(P_{\RR}^+\), \(P^++P^-\) and
\(P^+\&P^-\), as explained in Table~\ref{tab:projector}, to compute the GSVD
of {\tt{GL7d12}}.
The naive FEAST algorithm applied to the Jordan--Wielandt matrix is also
tested for comparison.
In order to make the difference easily visible, we replace the ellipse with a
circle and reduce the number of quadrature nodes to \(N=8\) in this test to
slow down convergence.

\begin{table}[!tb]
\centering
\caption{Four choices of spectral projectors.}
\label{tab:projector}
\begin{tabular}{cl}
\hline
Symbol & \multicolumn{1}{c}{Meaning} \\
\hline
\(P^+\) & the simple projector on \(\Gamma^+\)
(see Section~\ref{subsubsec:P^+}) \\
\(P_{\RR}^+\) & \(P^+\) with a Rayleigh--Ritz projection
(see Section~\ref{subsubsec:P^+_RR}) \\
\(P^++P^-\) & the projector combining \(\Gamma^+\) and \(\Gamma^-\)
(see Section~\ref{subsubsec:P^++P^-}) \\
\(P^+\&P^-\) & the augmented projector combining \(\Gamma^+\) and \(\Gamma^-\)
(see Section~\ref{subsubsec:P^+andP^-}) \\
\hline
\end{tabular}
\end{table}

With a randomly generated initial guess, \(P^+\&P^-\) demonstrates the best
convergence rate, and the other three behave similarly;
see Figure~\ref{fig:4filters}(a).
This is consistent with our discussions in Section~\ref{subsec:SVD-projector}.
In this case both \(P_{\RR}^+\) and \(P^++P^-\) are slightly more expensive
compared to~\(P^+\).
The naive FEAST algorithm is about twice slower than \(P^+\) because the naive
FEAST algorithm is computationally more expensive.

\begin{figure}[tb!]
\centering
\begin{tabular}{cc}
\includegraphics[width=0.4\textwidth]{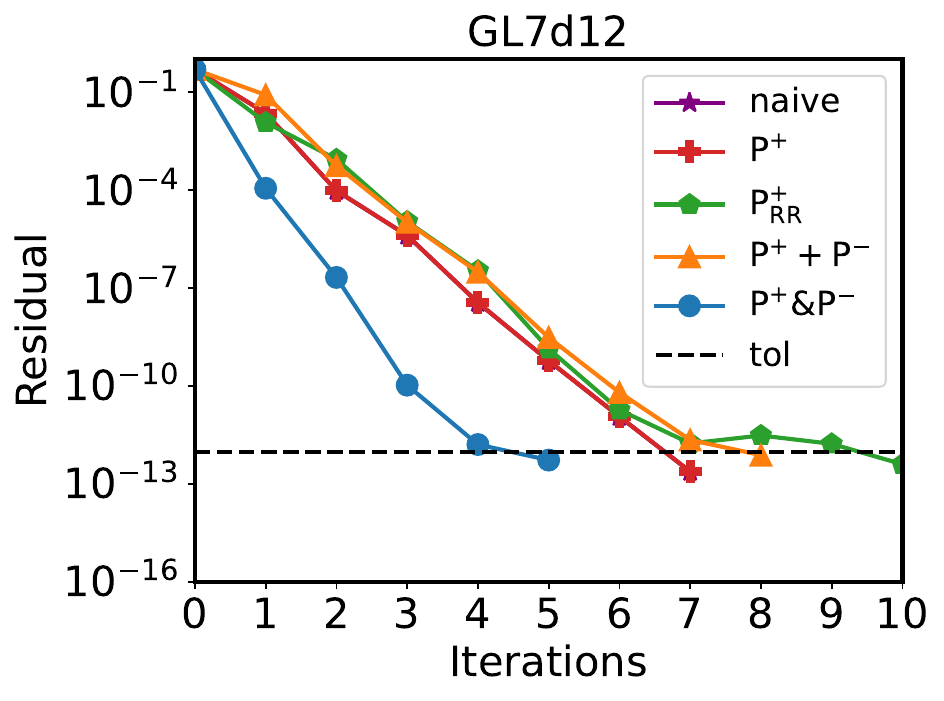} &
\includegraphics[width=0.4\textwidth]{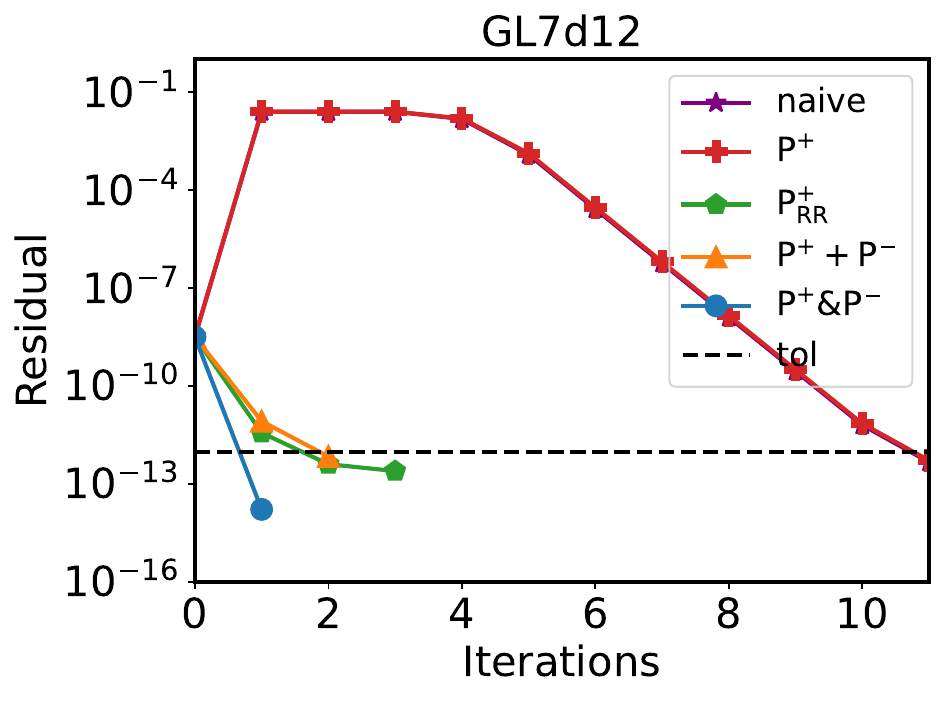}\\
\includegraphics[width=0.4\textwidth]{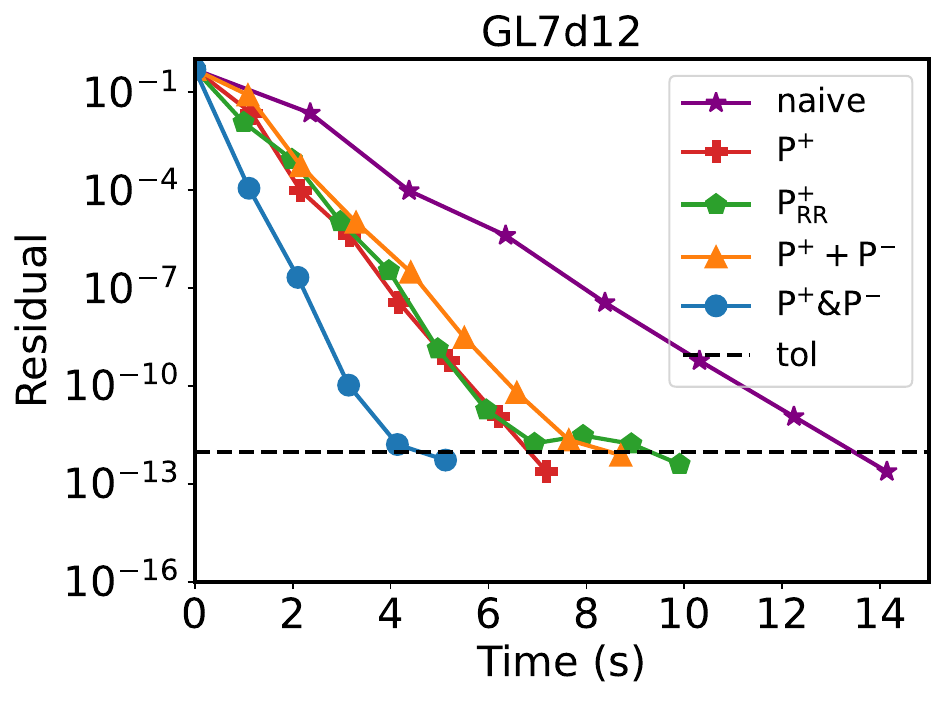} &
\includegraphics[width=0.4\textwidth]{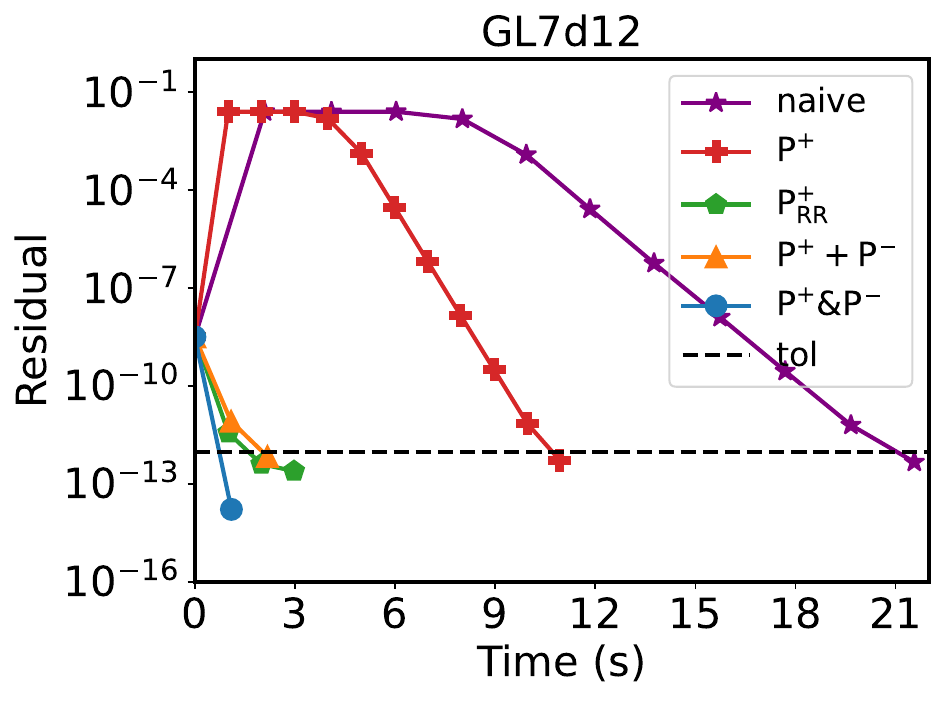}\\
(a) & (b) 
\end{tabular}
\caption{Comparison on different spectral projectors applied to compute the
GSVD of {\tt{GL7d12}}:\\
(a) using a random initial guess;
(b) using the artificial initial guess~\eqref{eq:artificial1}.}
\label{fig:4filters}
\end{figure}

We also construct an artificial initial guess
\begin{equation}
\label{eq:artificial1}
\left[\bmat{U\\-W},Q_{(m+n)\times(\ell-k)}\right]
+10^{-12}\sqrt{m}\cdot Q_{(m+n)\times\ell},
\end{equation}
where \([U\herm,W\herm]\herm\) is the desired solution, and the \(Q\)'s are
randomly generated matrices with orthonormal columns.
This initial guess already contains useful information of the true solution.
Figure~\ref{fig:4filters}(b) illustrates that \(P^+\) behaves poorly
if~\eqref{eq:artificial1} is used.
This example also supports our preference on \(P^+\&P^-\).

In Section~\ref{subsubsec:P^+andP^-}, we have mentioned that \(P^+\&P^-\) can
preserve a higher convergence rate for several steps even if the trial
subspace is only augmented in the first iteration.
Such an observation is also illustrated in Figure~\ref{fig:4filters}(a) ---
the quick convergence of its first iteration is inherited by a few subsequent
iterations.
Extra cost by augmenting the trial subspace in the first iteration is
compensated by rapid convergence.

\subsection{Iterative refinement}
We have seen that Algorithms~\ref{alg:FEAST-SVD} and~\ref{alg:FEAST-GSVD}
typically converge very rapidly.
These algorithms can naturally be adopted to refine low-precision solutions
produced by other eigensolvers.
To illustrate this, we generate a few artificial initial guesses of the form
\begin{equation}
\label{eq:artificial3}
\left[\bmat{U\\W},Q_{(m+n)\times(\ell-k)}\right]\cdot Q_{\ell\times\ell}
+10^{-q}\sqrt{m}\cdot Q_{(m+n)\times\ell},
\end{equation}
where \([U\herm,W\herm]\herm\) is the desired solution, and the \(Q\)'s are
randomly generated matrices with orthonormal columns.
The parameter \(q\) is chosen from \(\set{2,4,6,8,10,12}\) to control the
quality of the initial guess.

We choose the matrix \texttt{plat1919} and set the desired interval to
\((10^{-4},10^{-3})\), which contains~\(52\) generalized singular values.
This problem becomes relatively ill-conditioned as the desired generalized
singular values are small.
From Figure~\ref{fig:refinement}(a), we see that
Algorithm~\ref{alg:FEAST-GSVD} can refine the solution in two iterations
when \(q\geq6\).
In fact, it is possible to skip the second iteration because in this setting
the number of desired generalized singular values, \(\kGSVD\), is already
known, so that the additional stopping criterion in
Section~\ref{subsec:setting} can be avoided.
It is also possible to simply use the simple spectral projector \(P^+\) and
skip trace estimation, as the quality of the initial guess is known to be
good.

A more practical setting is to use the solution of MATLAB's
\({\tt{eigs}}(A\herm A, B\herm B)\) as the initial guess.
We see from Figure~\ref{fig:refinement}(b) that the accuracy is improved from
\(10^{-7}\) to \(10^{-13}\) in a single iteration.

\begin{figure}[tb!]
\centering
\begin{tabular}{cc}
\includegraphics[width=0.4\textwidth]{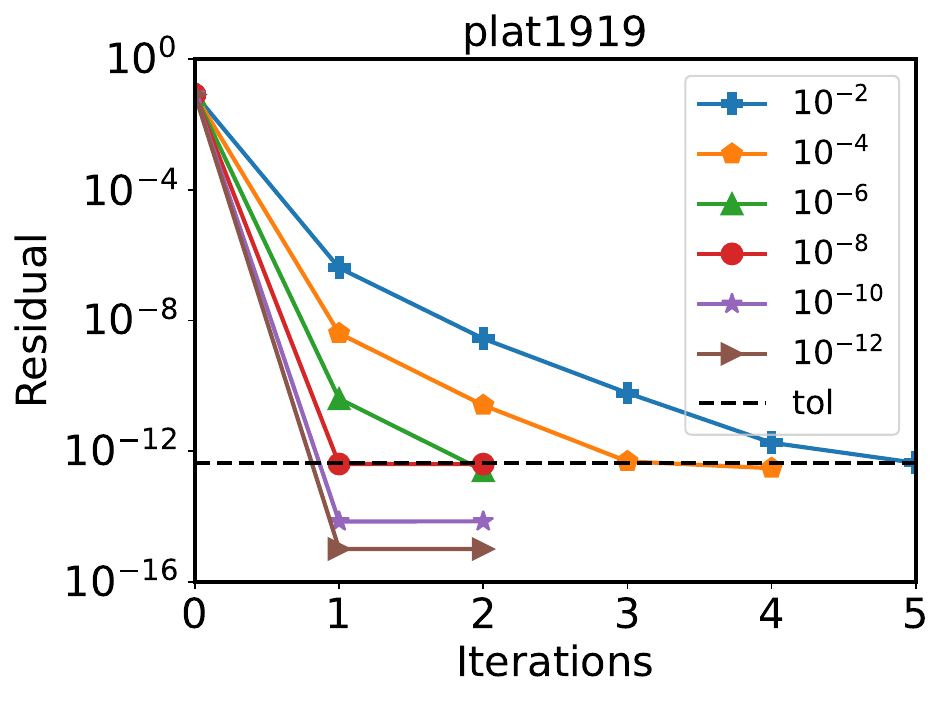} &
\includegraphics[width=0.4\textwidth]{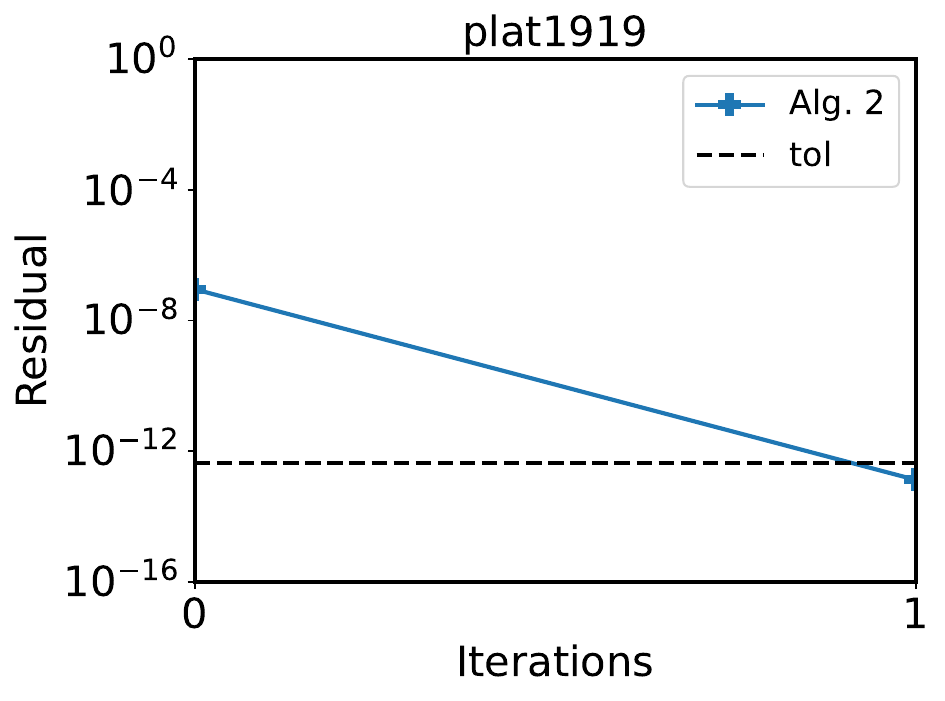}\\
(a) & (b)
\end{tabular}
\caption{Compute the generalized singular values of {\tt{plat1919}}
in the interval \((10^{-4},10^{-3})\) with low-precision initial guesses
by Algorithm~\ref{alg:FEAST-GSVD}:
(a) using artificial initial guesses~\eqref{eq:artificial3};
(b) using MATLAB's \({\tt{eigs}}(A\herm A, B\herm B)\) as the initial guess.}
\label{fig:refinement}
\end{figure}

\section{Conclusions}
\label{sec:conclusions}
In this work we propose a contour integral-based algorithm for computing
partial SVD/GSVD through the Jordan--Wielandt matrix (pencil).
This is a special case of the symmetric eigenvalue problem, targeting interior
eigenvalues.
We analyze four choices of spectral projectors tailored to this problem.
Though all of these choices lead to structure-exploiting algorithms, they may
behave quite differently, partly depending on the initial guess provided by
the user.
We identify one spectral projector that is both robust and effective.
Numerical experiments illustrate that our proposed algorithm 
can compute partial SVD/GSVD efficiently and accurately.

There are several potential usages of our algorithm.
When a large number of (generalized) singular values are of interest, our
algorithm can be incorporated into a spectral slicing framework.
Our algorithm can also be adopted to improve low-precision solutions produced
by other algorithms.
This is a promising feature that can possibly be exploited in modern mixed
precision algorithms.
Development in these directions is planned as our future work.

\appendix
\section{Further Discussions}
In Section~\ref{subsubsec:P^+andP^-}, we mentioned that augmenting the trial
subspace with a pair of contours can often accelerate the convergence of the
FEAST-SVD algorithm (in fact, also for FEAST-GSVD), and, in addition, such an
acceleration can be inherited by subsequent iterations even if the trial
subspace is only augmented in the first iteration.
In the following, we provide a brief explanation for such a fast convergence.

Let us assume that in the generic case
\[
\bmat{\tilde U \\ \tilde W}
=\frac{1}{\sqrt2}\bmat{U_\Inner & U_\Outer & U_\Inner & U_\Outer \\
W_\Inner & W_\Outer & -W_\Inner & -W_\Outer}
\bmat{C_\Inner^+ \\ C_\Outer^+ \\ C_\Inner^- \\ C_\Outer^-}
+\bmat{U_\NULL \\ 0}C_\NULL,
\]
and
\[
\tilde P^+(\check A)\bmat{\tilde U \\ \tilde W}
=\frac{1}{\sqrt2}\bmat{U_\Inner & U_\Outer & U_\Inner & U_\Outer\\
W_\Inner & W_\Outer & -W_\Inner & -W_\Outer}
\bmat{C_1 \\ C_2 \\ C_3 \\ C_4}
+\bmat{U_\NULL \\ 0}C_0,
\]
where \(C_1\approx C_\Inner^+\),
\(\max\set{\lVert C_0\rVert_2,\lVert C_2\rVert_2,\lVert C_3\rVert_2,
\lVert C_4\rVert_2}=O(\epsilon)\).
Note that
\[
\tilde P^+(\check A)\bmat{\tilde U \\ -\tilde W}\\
=\frac{1}{\sqrt2}\bmat{U_\Inner & U_\Outer & U_\Inner & U_\Outer\\
W_\Inner & W_\Outer & -W_\Inner & -W_\Outer}
\bmat{D_1 \\ D_2 \\ D_3 \\ D_4}
+\bmat{U_\NULL \\ 0}C_0,
\]
where \(D_1\approx C_\Inner^-\),
\(\max\set{\lVert D_2\rVert_2,\lVert D_3\rVert_2,\lVert D_4\rVert_2}
=O(\epsilon)\).
Then after one step of Algorithm~\ref{alg:FEAST-SVD} we obtain the coefficient
matrix
\begin{equation}
\label{eq:1-step}
\bmat{C_1 & D_1 \\ C_2 & D_2 \\ C_3 & D_3 \\ C_4 & D_4 \\ C_0 & C_0}
\approx\bmat{C_\Inner^+ & C_\Inner^- \\ O(\epsilon) & O(\epsilon) \\
O(\epsilon) & O(\epsilon) \\ O(\epsilon) & O(\epsilon) \\
O(\epsilon) & O(\epsilon)}
\in\mathbb C^{(m+n)\times2\ell}.
\end{equation}
If, instead, we apply the simple spectral projector to a generic initial guess
with \(2\ell\) columns, the corresponding coefficient matrix becomes
\[
\bmat{C_1 \\ C_2 \\ C_3 \\ C_4 \\ C_0}
\approx\bmat{C_\Inner^+ \\ O(\epsilon) \\ O(\epsilon) \\ O(\epsilon) \\
O(\epsilon)}
\in\mathbb C^{(m+n)\times2\ell},
\]
which is about equally good compared to~\eqref{eq:1-step}.
Since the FEAST algorithm is essentially subspace iteration on the
spectral projector, it remains explaining why one step of subspace iteration
with an augmented initial guess can accelerate subsequent iterations.

In the rest of this section, we consider an \(n\times n\) Hermitian matrix
\(A\) whose eigenvalues \(\lambda_1\), \(\dotsc\), \(\lambda_n\) are ordered
such that
\[
\lvert\lambda_{1}\rvert\geq\lvert\lambda_{2}\rvert\geq\cdots
\geq\lvert\lambda_{k}\rvert>\lvert\lambda_{k+1}\rvert\geq\cdots
\geq\lvert\lambda_{\ell}\rvert>\lvert\lambda_{\ell+1}\rvert\geq\cdots\geq
\lvert\lambda_{n}\rvert.
\]
Let \(U\) represent the matrix containing  eigenvectors
\(u_1\), \(\dotsc\), \(u_n\).
The eigenvalues of interest are \(\lambda_1\), \(\dotsc\), \(\lambda_k\).
The initial guess contains \(\ell\) columns.
We assume that
\begin{equation}
\label{eq:gap}
\lvert\lambda_1\rvert=\Theta(1),
\qquad \lvert\lambda_k\rvert=\Theta(1),
\qquad \lvert\lambda_{\ell+1}\rvert=O(\epsilon).
\end{equation}
The spectral projector arising in the FEAST algorithm usually
satisfies~\eqref{eq:gap}.

Convergence analysis of subspace iteration can be found in many textbooks;
see, e.g., \cite{Parlett1998}.
Classical results focus on the asymptotic convergence rate, while we are
interested in convergence at the early stage.
To this end we establish the following technical results.

\begin{lemma}
\label{lem:0}
Let \(M\in\mathbb C^{m\times m}\) be Hermitian and positive semidefinite.
Then
\[
\lVert(I_m+M)^{-1/2}-I_m\rVert_2\leq\frac12\lVert M\rVert_2.
\]
\end{lemma}

\begin{proof}
Let \(\mu_1\), \(\mu_2\), \(\dotsc\), \(\mu_m\) be the eigenvalues of \(M\).
Then
\begin{align*}
\lVert(I_m+M)^{-1/2}-I_m\rVert_2
&=\max_{1\leq i\leq m}\bigl\lvert(1+\mu_i)^{-1/2}-1\bigr\rvert \\
&=\max_{1\leq i\leq m}\frac{\mu_i}{(1+\mu_i)^{1/2}+(1+\mu_i)} \\
&\leq\frac12\max_{1\leq i\leq m}\mu_i \\
&=\frac12\lVert M\rVert_2.
\qedhere
\end{align*}
\end{proof}

\begin{theorem}
\label{thm:1}
Let \(A\in\mathbb C^{n\times n}\) be a nonsingular Hermitian matrix with
normalized eigenpairs \((\lambda_1,u_1)\), \((\lambda_2,u_2)\), \(\dotsc\),
\((\lambda_n,u_n)\).
Assume that \(\lvert\lambda_1\rvert\geq\lvert\lambda_2\rvert\geq\dotsb
\geq\lvert\lambda_n\rvert>0\).
A matrix \(X\in\mathbb C^{n\times \ell}\) is of the form
\[
X=[U_\ell,U_\ell^{\perp}]\bmat{X_1\\X_2},
\]
where \(U_\ell=[u_1,\dotsc,u_\ell]\),
\(U_\ell^{\perp}=[u_{\ell+1},\dotsc,u_n]\), and
\(X_1\in\mathbb C^{\ell\times\ell}\) is nonsingular.
Let \(X_3=X_2X_1^{-1}\) be partitioned into \(X_3=[X_{3,1},X_{3,2}]\), where
\(X_{3,1}\in\mathbb C^{(n-\ell)\times k}\).
Then there exists \(Y\in\mathbb C^{n\times\ell}\) such that
\(\Span(Y)=\Span(AX)\), \(Y\herm Y=I_\ell\), and
\[
Y=U_\ell+[U_k,U_{\ell\backslash k},U_\ell^{\perp}]
\begin{bNiceMatrix}[first-row,last-col]
k & \ell-k & \\
E_{1,1} & E_{1,2} & ~k \\
E_{2,1} & E_{2,2} & ~\ell-k \\
F_1 & F_2 & ~n-\ell
\end{bNiceMatrix},
\]
for \(1\leq k\leq\ell\), where \(U_k=[u_1,\dotsc,u_k]\) and
\(U_{\ell\backslash k}=[u_{k+1},\dotsc,u_{\ell}]\).
Define
\[
\tilde\eta=\bigl\lVert\Lambda_\ell^{\perp}X_{3,1}\Lambda_k^{-1}
\bigr\rVert_2,
\qquad
\hat\eta=\bigl\lVert\Lambda_\ell^{\perp}X_{3,2}\Lambda_{\ell\backslash k}^{-1}
\bigr\rVert_2,
\]
where \(\Lambda_k=\diag\set{\lambda_1,\lambda_2,\dotsc,\lambda_k}\),
\(\Lambda_{\ell\backslash k}=\diag\set{\lambda_{k+1},\lambda_{k+2},\dotsc,\lambda_\ell}\) and
\(\Lambda_\ell^{\perp}
=\diag\set{\lambda_{\ell+1},\lambda_{\ell+2},\dotsc,\lambda_n}\).
Then we have
\begin{align*}
\lVert E_{1,1}\rVert_2&\leq\frac12\tilde\eta^2, &
\lVert E_{2,1}\rVert_2&=0, &
\lVert F_{1}\rVert_2&\leq\tilde\eta, \\
\lVert E_{1,2}\rVert_2&\leq\tilde\eta\hat\eta, &
\lVert E_{2,2}\rVert_2&\leq\frac12\hat\eta^2, &
\lVert F_{2}\rVert_2&\leq\hat\eta.
\end{align*}
\end{theorem}

\begin{proof}
Let \(\Lambda=\diag\set{\Lambda_k,
\Lambda_{\ell\backslash k},\Lambda_\ell^{\perp}}\).
We express \(AX\) as
\[
AX=A[U_\ell,U_\ell^{\perp}]\bmat{X_1 \\ X_2}
=[U_\ell,U_\ell^{\perp}]\Lambda\bmat{I_\ell \\ X_2X_1^{-1}}X_1
=[U_\ell,U_\ell^{\perp}]B\Lambda_\ell X_1,
\]
where
\begin{equation*}
\label{eq:B}
B=\bmat{I_\ell \\ \Lambda_\ell^{\perp}X_3\Lambda_\ell^{-1}}.
\end{equation*}
Partition \(B\) into \([B_k,B_{\ell\backslash k}]\), where \(B_k\) is \(n\times k\).
Let \(M_k=\Lambda_k^{-1}X_{3,1}\herm(\Lambda_\ell^{\perp})^2X_{3,1}\Lambda_k^{-1}\),
\(E_{1,1}=(B_k\herm B_k)^{-1/2}-I_k\).
Then \(B_k\) can be orthonormalized through
\[
Q_k=B_k(B_k\herm B_k)^{-1/2}
=\bmat{I_k \\ 0 \\ \Lambda_\ell^{\perp}X_{3,1}\Lambda_k^{-1}}
(I_k+E_{1,1})
=\bmat{I_k \\ 0 \\ \Lambda_\ell^{\perp}X_{3,1}\Lambda_k^{-1}}
\bigl(I_k+M_k\bigr)^{-1/2},
\]
where \(\lVert E_{1,1}\rVert_2\) can be bounded by Lemma~\ref{lem:0} as
\[
\lVert E_{1,1}\rVert_2
\leq\frac12\lVert M_k\rVert_2
\leq\frac12\tilde\eta^2.
\]
Let
\[
F_1=\Lambda_\ell^{\perp}X_{3,1}\Lambda_k^{-1}(I_k+E_{1,1}).
\]
Notice that \([I_k+E_{1,1}\herm,0,F_1\herm]\herm\) has orthonormal columns.
By CS decomposition we obtain
\begin{align}
\label{eq:proof_CS}
\begin{aligned}
\lVert F_1\rVert_2
&=\bigl(1-\sigma_{\min}(I_k+E_{1,1})^2\bigr)^{1/2}\\
&=\bigl(1-\sigma_{\min}\bigl((I_k+M_k)^{-1}\bigr)\bigr)^{1/2}\\
&=\bigl(1-\lVert I_k+M_k\rVert_2^{-1}\bigr)^{1/2}\\
&=\lVert M_k\rVert_2^{1/2}(1+\lVert M_k\rVert_2)^{-1/2}\\
&\leq\tilde\eta.
\end{aligned}
\end{align}
We then orthonormalize \(B_{\ell\backslash k}\) against \(Q_k\)
through
\[
C_{\ell\backslash k}
=(I_n-Q_kQ_k\herm)B_{\ell\backslash k}
=\bmat{-(I_k+E_{1,1})F_1\herm
\Lambda_\ell^{\perp}X_{3,2}\Lambda_{\ell\backslash k}^{-1} \\
I_{\ell-k} \\
(I_{n-\ell}-F_1F_1\herm)
\Lambda_\ell^{\perp}X_{3,2}\Lambda_{\ell\backslash k}^{-1}}
\]
and
\[
Q_{\ell\backslash k}
=C_{\ell\backslash k}
(C_{\ell\backslash k}\herm C_{\ell\backslash k})^{-1/2}
=C_{\ell\backslash k}\bigl(I_{\ell-k}
+M_{\ell\backslash k}\bigr)^{-1/2}
=C_{\ell\backslash k}(I_{\ell-k}+E_{2,2}),
\]
where
\[
M_{\ell\backslash k}
=C_{\ell\backslash k}\herm C_{\ell\backslash k}-I_{\ell-k}
=\Lambda_{\ell\backslash k}^{-1}X_{3,2}\herm\Lambda_\ell^{\perp}
(I_{n-\ell}-F_1F_1\herm)
\Lambda_\ell^{\perp}X_{3,2}\Lambda_{\ell\backslash k}^{-1}
\]
and
\(E_{2,2}=(C_{\ell\backslash k}\herm
C_{\ell\backslash k})^{-1/2}-I_{\ell-k}\).
By Lemma~\ref{lem:0} we have
\[
\lVert E_{2,2}\rVert_2
\leq\frac12\lVert M_{\ell\backslash k}\rVert_2
\leq\frac12\lVert I_{n-\ell}-F_1F_1\herm\rVert_2
\lVert\Lambda_\ell^{\perp}X_{3,2}\Lambda_{\ell\backslash k}^{-1}\rVert_2^2
\leq\frac12\hat\eta^2.
\]
Let
\(Q=[Q_k,Q_{\ell\backslash k}]\)
and
\[
Y=[U_k,U_{\ell\backslash k},U_{\ell}^{\perp}]Q
=U_\ell+[U_k,U_{\ell\backslash k},U_{\ell}^{\perp}]
\bmat{E_{1,1} & E_{1,2} \\ 0 & E_{2,2} \\ F_1 & F_2},
\]
where \(E_{1,2}\) and \(F_2\) satisfy
\[
\bmat{E_{1,2} \\ E_{2,2} \\ F_2}
=Q_{\ell\backslash k}-\bmat{0 \\ I_{\ell-k} \\ 0}.
\]
Then \(Y\herm Y=I_\ell\) and \(\Span(Y)=\Span(AX)\).
Notice that \([E_{1,2}\herm,I_{\ell-k}+E_{2,2}\herm,F_2\herm]\herm\) has
orthonormal columns.
Similar to the proof of \eqref{eq:proof_CS}, we obtain
\[
\lVert F_2\rVert_2
\leq\left\lVert\bmat{E_{1,2} \\ F_2}\right\rVert_2
=\lVert M_{\ell\backslash k}\rVert_2^{1/2}
(1+\lVert M_{\ell\backslash k}\rVert_2)^{-1/2}
\leq\hat\eta
\]
and
\[
\lVert E_{1,2}\rVert_2
\leq\lVert I_k+E_{1,1}\rVert_2\lVert F_1\rVert_2
\lVert\Lambda_\ell^{\perp}X_{3,2}\Lambda_{\ell\backslash k}^{-1}\rVert_2
\lVert I_{\ell-k}+E_{2,2}\rVert_2
\leq\lVert F_1\rVert_2
\lVert\Lambda_\ell^{\perp}X_{3,2}\Lambda_{\ell\backslash k}^{-1}\rVert_2
\leq\tilde\eta\hat\eta.
\qedhere
\]
\end{proof}

\begin{theorem}
\label{thm:perturbation}
Let \(H=\Lambda+\Delta H\in\mathbb C^{\ell\times \ell}\) be a Hermitian matrix with
spectral decomposition \(H=Q\Theta Q\herm\), where \(\Lambda\) and \(\Theta\)
are real diagonal matrices, and \(Q\) is unitary.
Partition \(\Lambda\), \(\Theta\) and \(\Delta H\) into
\(\Lambda=\diag\set{\Lambda_k,\Lambda_{\ell\backslash k}}\),
\(\Theta=\diag\set{\Theta_k,\Theta_{\ell\backslash k}}\), and
\(\Delta H=[\Delta H_1,\Delta H_2]\),
where \(\Lambda_k\), \(\Theta_k\in\mathbb{R}^{k\times k}\)
and \(\Delta H_{1}\in\mathbb{C}^{\ell\times k}\).
Suppose
\[
\spec(\Lambda_k)\subset[\alpha,\beta],
\qquad\spec(\Theta_{\ell\backslash k})\subset\mathbb R\backslash(\alpha-\delta,\beta+\delta),
\]
where \(\delta>0\).
Then there exist unitary matrices \(Q_1\in\mathbb C^{k\times k}\)
and \(Q_2\in\mathbb C^{(\ell-k)\times(\ell-k)}\) satisfying
\[
Q=\bmat{Q_1 \\ & Q_2}
+\bmat{\Delta_{1,1} & \Delta_{1,2} \\ \Delta_{2,1} & \Delta_{2,2}},
\qquad \lVert\Delta_{ij}\rVert_2\leq
\begin{cases}
\epsilon_\delta^2, & i=j,\\
\epsilon_\delta, & i\neq j,
\end{cases}
\]
in which \(\epsilon_\delta=\lVert\Delta H_1\rVert_2/\delta\).
\end{theorem}

\begin{proof}
Partition \(Q\) into
\[
Q=\bmat{Q_{1,1} & Q_{1,2} \\ Q_{2,1} & Q_{2,2}},
\]
where \(Q_{1,1}\in\mathbb C^{k\times k}\).
According to the Davis--Kahan \(\sin\theta\) theorem~\cite{DK1970}, it can be
verified that
\[
\lVert Q_{1,2}\rVert_2\leq\frac{\lVert\Delta H_{1}\rVert_2}{\delta}=\epsilon_\delta.
\]
Notice that \(\lVert Q_{1,2}\rVert_2=\lVert Q_{2,1}\rVert_2\).
Using~\cite[Lemma~5.1]{Higham1994}, we know that there is a unitary matrix
\(Q_1\in\mathbb{C}^{k\times k}\) satisfying
\[
\lVert Q_{1,1}-Q_1\rVert_2
\leq\lVert I_k-Q_{1,1}\herm Q_{1,1}\rVert_2
=\lVert Q_{2,1}\herm Q_{2,1}\rVert_2
\leq\epsilon_\delta^2.
\]
Similarly, there exists a unitary matrix
\(Q_2\in\mathbb{C}^{(\ell-k)\times(\ell-k)}\) such that
\[
\lVert Q_{2,2}-Q_2\rVert_2
\leq\lVert I_{\ell-k}-Q_{2,2}\herm Q_{2,2}\rVert_2
=\lVert Q_{1,2}\herm Q_{1,2}\rVert_2
\leq\epsilon_\delta^2.
\]
Setting
\[
\bmat{\Delta_{1,1} & \Delta_{1,2} \\ \Delta_{2,1} & \Delta_{2,2}}
=Q-\bmat{Q_1 \\ & Q_2}
\]
yields the conclusion.
\end{proof}

\begin{theorem}
\label{thm:rate}
Let \(A\) be Hermitian with spectral decomposition \(A=U\Lambda U\herm\).
Suppose
\[
X=U\bmat{X_k \\ X_{\ell\backslash k} \\ X_\ell^{\perp}}
=[U_k,U_{\ell\backslash k},U_\ell^{\perp}]
\bmat{X_k \\ X_{\ell\backslash k} \\ X_\ell^{\perp}}
\in\mathbb C^{n\times k},
\]
where \(X_k\in\mathbb C^{k\times k}\) is nonsingular,
\(X_{\ell\backslash k}\in\mathbb C^{(\ell-k)\times k}\).
Then
\begin{equation}
\label{eq:rate}
\tan\angle(U_k,AX)
\leq\frac{\lvert\lambda_{\ell+1}\rvert}{\lvert\lambda_k\rvert}
\tan\angle(U_k,X)
+\frac{\bigl\lVert\Lambda_{\ell\backslash k}X_{\ell\backslash k}
X_k^{-1}\bigr\rVert_2}{\lvert\lambda_k\rvert}.
\end{equation}
\end{theorem}

\begin{proof}
We have
\[
\tan\angle(U_k,X)
=\left\lVert\bmat{X_{\ell\backslash k}X_k^{-1} \\
X_\ell^{\perp}X_k^{-1}}\right\rVert_2
=\left\lVert\bmat{X_{\ell\backslash k} \\
X_\ell^{\perp}}X_k^{-1}\right\rVert_2
\]
because
\[
X=U\bmat{X_k \\ X_{\ell\backslash k} \\ X_\ell^{\perp}}
=U\bmat{I_k \\
X_{\ell\backslash k}X_k^{-1} \\
X_\ell^{\perp}X_k^{-1}}X_k.
\]
Applying \(A\) to \(X\) yields
\[
AX=U\Lambda\bmat{X_k \\ X_{\ell\backslash k} \\ X_\ell^{\perp}}
=U\bmat{I_k \\
\Lambda_{\ell\backslash k}X_{\ell\backslash k}
X_k^{-1}\Lambda_k^{-1} \\
\Lambda_\ell^{\perp}X_\ell^{\perp}X_k^{-1}\Lambda_k^{-1}}
\Lambda_kX_k
\]
and
\[
\tan\angle(U_k,AX)
=\left\lVert\bmat{\Lambda_{\ell\backslash k}X_{\ell\backslash k}
X_k^{-1}\Lambda_k^{-1} \\
\Lambda_\ell^{\perp}X_\ell^{\perp}X_k^{-1}\Lambda_k^{-1}}
\right\rVert_2
\leq\frac{1}{\lvert\lambda_k\rvert}
\left\lVert\bmat{\Lambda_{\ell\backslash k}X_{\ell\backslash k} \\
\Lambda_\ell^{\perp}X_\ell^{\perp}}X_k^{-1}\right\rVert_2.
\]
Notice that
\begin{align*}
\left\lVert\bmat{\Lambda_{\ell\backslash k}X_{\ell\backslash k} \\
\Lambda_\ell^{\perp}X_\ell^{\perp}}
X_k^{-1}\right\rVert_2
&\leq\left(\bigl\lVert\Lambda_{\ell\backslash k}X_{\ell\backslash k}
X_k^{-1}\bigr\rVert_2^2
+\bigl\lVert\Lambda_\ell^{\perp}X_\ell^{\perp}
X_k^{-1}\bigr\rVert_2^2\right)^{1/2} \\
&\leq\left(\bigl\lVert\Lambda_{\ell\backslash k}X_{\ell\backslash k}
X_k^{-1}\bigr\rVert_2^2
+\lambda_{\ell+1}^2\left\lVert\bmat{X_{\ell\backslash k} \\ X_\ell^{\perp}}
X_k^{-1}\right\rVert_2^2\right)^{1/2} \\
&\leq\bigl\lVert\Lambda_{\ell\backslash k}X_{\ell\backslash k}
X_k^{-1}\bigr\rVert_2
+\lvert\lambda_{\ell+1}\rvert
\left\lVert\bmat{X_{\ell\backslash k} \\ X_\ell^{\perp}}
X_k^{-1}\right\rVert_2.
\end{align*}
Therefore, we have
\[
\tan\angle(U_k,AX)
\leq\frac{\lvert\lambda_{\ell+1}\rvert}{\lvert\lambda_k\rvert}
\tan\angle(U_k,X)
+\frac{\bigl\lVert\Lambda_{\ell\backslash k}X_{\ell\backslash k}
X_k^{-1}\bigr\rVert_2}{\lvert\lambda_k\rvert}.
\qedhere
\]
\end{proof}

We aim at computing \(k\) leading eigenvalues of \(A\).
First, we act \(A\) on the initial matrix \(X\in\mathbb{C}^{n\times\ell}\),
where \(\ell>k\).
According to Theorem~\ref{thm:1}, we
know that there exists
\(Y=U_\ell(I+E)+U_\ell^{\perp}F\in\mathbb C^{n\times \ell}\) such that
\(Y\herm Y=I_\ell\) and \(\Span(Y)=\Span(AX)\),
where
\[
E=\bmat{E_{1,1} & E_{1,2} \\ E_{2,1} & E_{2,2}},
\quad F=\bmat{F_1 & F_2}.
\]
The norm of each block of \(E\) and \(F\) can also be estimated by
Theorem~\ref{thm:1}.

The projected matrix \(Y\herm AY\) can be written as
\[
Y\herm AY=\Lambda_\ell+E\herm\Lambda_\ell+\Lambda_\ell E
+\bmat{E \\ F}\herm\Lambda\bmat{E \\ F}
=\Lambda_\ell+\Delta H,
\]
i.e., \(Y\herm AY\) can be regarded as a diagonal matrix \(\Lambda_\ell\) with
a Hermitian perturbation \(\Delta H\).
The leading \(k\) columns of \(\Delta H\), denoted as \(\Delta H_1\),
is bounded through
\begin{align*}
\lVert\Delta H_1\rVert_2
&=\left\lVert
\bmat{E_{1,1}\herm\Lambda_k \\ E_{1,2}\herm\Lambda_k}
+\bmat{\Lambda_k E_{1,1} \\ \Lambda_{\ell\backslash k}E_{2,1}}
+\bmat{E_{1,1}\herm\Lambda_kE_{1,1}
+E_{2,1}\herm\Lambda_{\ell\backslash k}E_{2,1}
+F_1\herm\Lambda_\ell^{\perp}F_1 \\
E_{1,2}\herm\Lambda_kE_{1,1}
+E_{2,2}\herm\Lambda_{\ell\backslash k}E_{2,1}
+F_2\herm \Lambda_\ell^{\perp}F_1}
\right\rVert_2\\
&\leq\lvert\lambda_1\rvert(2\lVert E_{1,1}\rVert_2
+\lVert E_{1,2}\rVert_2+\lVert E_{1,1}\rVert_2^2
+\lVert E_{1,1}\rVert_2\lVert E_{1,2}\rVert_2)
+\lvert\lambda_{\ell+1}\rvert(\lVert F_1\rVert_2^2
+\lVert F_1\rVert_2\lVert F_2\rVert_2)\\
&=O\bigl(\lvert\lambda_1\rvert\cdot\lVert E_{1,2}\rVert_2
+\lvert\lambda_{\ell+1}\rvert\cdot\lVert F_1\rVert_2\lVert F_2\rVert_2\bigr).
\end{align*}
Suppose that the conditions of Theorem~\ref{thm:perturbation} hold, and the
upper bound of value \(\epsilon_\delta\) is
\[
\epsilon_\delta=\frac{\lVert\Delta H_1\rVert_{2}}{\delta}\leq
\frac{\lvert\lambda_1\rvert}{\delta}\cdot O(\hat\eta\tilde\eta).
\]
After the Rayleigh--Ritz projection \(Y\herm AY=Q\Theta Q\herm\), the
approximate eigenvectors \(X\) become
\begin{equation*}
\label{RR:X}
X=YQ=U_\ell(I_\ell+E)Q+U_\ell^{\perp}FQ.
\end{equation*}
Then we have
\[
(I_\ell+E)Q=\bmat{(I_k+E_{1,1})(Q_1+\Delta_{1,1})+E_{1,2}\Delta_{2,1} & * \\
E_{2,1}(Q_{1}+\Delta_{1,1})+(I_{\ell-k}+E_{2,2})\Delta_{2,1} & *},
\]
and
\[
FQ=\bmat{F_1(Q_1+\Delta_{1,1})+F_2\Delta_{2,1} & * },
\]
where
\begin{equation}
\label{eq:X1}
\lVert(I_k+E_{1,1})(Q_1+\Delta_{1,1})+E_{1,2}\Delta_{2,1}-Q_1\rVert_2
\leq\lVert E_{1,1}\rVert_2+\lVert E_{1,2}\rVert_2\lVert \Delta_{2,1}\rVert_2
=O(\tilde\eta^2),
\end{equation}
\begin{equation}
\label{eq:X2}
\begin{aligned}
\lVert E_{2,1}(Q_{1}+\Delta_{1,1})+(I_{\ell-k}+E_{2,2})\Delta_{2,1}\rVert_2
=\lVert(I_{\ell-k}+E_{2,2})\Delta_{2,1}\rVert_2
\leq\lVert\Delta_{2,1}\rVert_2
=O\biggl(\frac{\lvert\lambda_1\rvert}{\delta}\hat\eta\tilde\eta\biggr),
\end{aligned}
\end{equation}
and
\begin{align}
\begin{aligned}
\label{eq:X3}
\lVert F_1(Q_1+\Delta_{1,1})+F_2\Delta_{2,1}\rVert_2
\leq\lVert F_1\rVert_2+\lVert F_2\rVert_2\lVert\Delta_{21}\rVert_2
=O(\tilde\eta).
\end{aligned}
\end{align}
Since the diagonal entries of \(\Lambda_{\ell\backslash k}\) decay,
\(\hat\eta\) is usually much smaller than the most pessimistic upper bound
\(\lvert\lambda_{\ell+1}\rvert/\lvert\lambda_\ell\rvert\cdot\lVert X_{32}\rVert_2\).
Then~\eqref{eq:X3} is typically larger than~\eqref{eq:X2} when
\(\lvert\lambda_1\rvert/\delta=O(1)\).
We remark that \(\lvert\lambda_1\rvert/\delta=O(1)\) is not unusual in our
original setting of contour integral-based solvers, as the gap \(\delta\)
depends on the distance from the contour to the closest unwanted eigenvalue.

Next, we update \(X\) by keeping its leading \(k\) columns and dropping the
trailing ones.
Then
\[
X=[U_k,U_{\ell\backslash k},U_\ell^{\perp}]\bmat{X_k \\ X_{\ell\backslash k} \\ X_\ell^{\perp}},
\]
where \(X_{\ell\backslash k}\) and \(X_\ell^{\perp}\) are bounded
by~\eqref{eq:X2} and~\eqref{eq:X3}, respectively.
By~\eqref{eq:rate} in Theorem~\ref{thm:rate}, we have
\begin{equation}
\label{eq:rate3}
\frac{\tan\angle(U_k,AX)}{\tan\angle(U_k,X)}
\leq\frac{\lvert\lambda_{\ell+1}\rvert}{\lvert\lambda_k\rvert}
+\kappa(X_k)
\frac{\lVert\Lambda_{\ell\backslash k}X_{\ell\backslash k}\rVert_2}
{\lvert\lambda_k\rvert\lVert X_\ell^{\perp}\rVert_2}.
\end{equation}
We conclude that \(\kappa(X_k)=\Theta(1)\) according to~\eqref{eq:X1}, and
\(\lVert X_\ell^{\perp}\rVert_2\) is in general larger than
\(\lVert X_{\ell\backslash k}\rVert_2\).
In practice, \(\lvert\lambda_k\rvert\lVert X_\ell^{\perp}\rVert_2\) is even
much larger than
\(\lVert\Lambda_{\ell\backslash k}X_{\ell\backslash k}\rVert_2\),
because the diagonal entries of \(\Lambda_{\ell\backslash k}\) decay rapidly.
In fact, from numerical experiments we observe that
\(\lvert\lambda_{\ell+1}\rvert/\lvert\lambda_k\rvert\) dominates the
right-hand-side of~\eqref{eq:rate3}.
Therefore, the local convergence rate is roughly equal to
\(\lvert\lambda_{\ell+1}\rvert/\lvert\lambda_k\rvert=O(\epsilon)\).
After a few iterations, the local convergence rate gradually deteriorates,
and eventually returns to the asymptotic convergence rate
\(\lvert\lambda_{k+1}\rvert/\lvert\lambda_k\rvert\).

\section*{Acknowledgments}
The authors thank Zhaojun Bai, Weiguo Gao, Zhongxiao Jia, Daniel Kressner,
Yingzhou Li, and Jose E. Roman for helpful discussions.
This work is partially supported by the National Natural Science Foundation
of China under grant No.~92370105.

\addcontentsline{toc}{section}{References}

\end{document}